\documentclass{amsart} 
\usepackage{amssymb,latexsym,amsfonts,verbatim,amscd,ifthen} 
\usepackage{color}

\numberwithin{equation}{section}

\theoremstyle{plain}

\newtheorem{theorem}[equation]{Theorem}   
\newtheorem{lemma}[equation]{Lemma} 
\newtheorem{proposition}[equation]{Proposition} 
\newtheorem{corollary}[equation]{Corollary} 
\newtheorem{problem}[equation]{Problem}

\theoremstyle{definition}

\newtheorem{definition}[equation]{Definition} 
\newtheorem{remark}[equation]{Remark} 
\newtheorem{remarks}[equation]{Remarks}
 
\newtheorem{example}[equation]{Example}
\newtheorem{convention}[equation]{Convention}

\DeclareMathOperator{\HH}{H}
\DeclareMathOperator{\RH}{\widetilde{H}}

\DeclareMathOperator{\id}{id}  
\DeclareMathOperator{\coker}{Coker}

\DeclareMathOperator{\rank}{rank}

\DeclareMathOperator{\im}{Im}

\DeclareMathOperator{\Ker}{Ker} 
\DeclareMathOperator{\lcm}{lcm} 
\DeclareMathOperator{\comp}{{\bf Comp}}

\begin{document}   

\renewcommand{\:}{\! :} 
\newcommand{\ds}{\displaystyle}
\newcommand{\cx}{\mathfrak}
\newcommand{\mc}{\mathcal}
\newcommand{\p}{\mathfrak p} 
\newcommand{\m}{\mathfrak m}
\newcommand{\e}{\epsilon}
\newcommand{\lra}{\longrightarrow} 
\newcommand{\lla}{\longleftarrow}
\newcommand{\ra}{\rightarrow} 
\newcommand{\altref}[1]{{\upshape(\ref{#1})}} 
\newcommand{\bfa}{\boldsymbol{\alpha}} 
\newcommand{\bfb}{\boldsymbol{\beta}} 
\newcommand{\bfg}{\boldsymbol{\gamma}} 
\newcommand{\bfM}{\mathbf M} 
\newcommand{\bfI}{\mathbf I} 
\newcommand{\bfC}{\mathbf C} 
\newcommand{\bfB}{\mathbf B} 
\newcommand{\bsfC}{\bold{\mathsf C}} 
\newcommand{\bsfT}{\bold{\mathsf T}}
\newcommand{\smsm}{\smallsetminus} 
\newcommand{\ol}{\overline} 
\newcommand{\hm}{\hphantom{-}}

\newlength{\wdtha}
\newlength{\wdthb}
\newlength{\wdthc}
\newlength{\wdthd}
\newcommand{\elabel}[1]
           {\label{#1}  
            \setlength{\wdtha}{.4\marginparwidth}
            \settowidth{\wdthb}{\tt\small{#1}} 
            \addtolength{\wdthb}{\wdtha}
            \raisebox{\baselineskip}
            {\color{red} 
             \hspace*{-\wdthb}\tt\small{#1}\hspace{\wdtha}}}  

\newcommand{\mlabel}[1] 
           {\label{#1} 
            \setlength{\wdtha}{\textwidth}
            \setlength{\wdthb}{\wdtha} 
            \addtolength{\wdthb}{\marginparsep} 
            \addtolength{\wdthb}{\marginparwidth}
            \setlength{\wdthc}{\marginparwidth}
            \setlength{\wdthd}{\marginparsep}
            \addtolength{\wdtha}{2\wdthc}
            \addtolength{\wdtha}{2\marginparsep} 
            \setlength{\marginparwidth}{\wdtha}
            \setlength{\marginparsep}{-\wdthb} 
            \setlength{\wdtha}{\wdthc} 
            \addtolength{\wdtha}{1.4ex} 
            \settowidth{\wdthb}{\tt\small{#1}} 
            \marginpar{\vspace*{\baselineskip}
            \smash{\raisebox{0.7\baselineskip}{\tt\small{#1}}
            \hspace{-\wdthb}%
            \raisebox{.3\baselineskip}
            {\rule{\wdtha}{0.5pt}}} }
            \setlength{\marginparwidth}{\wdthc} 
            \setlength{\marginparsep}{\wdthd}  
            }  
            
\renewcommand{\mlabel}{\label}
\renewcommand{\elabel}{\label}


\title[Dynamical systems on chain complexes]
{Dynamical systems on chain complexes \\    
and canonical minimal resolutions}
\author[A. Tchernev]{Alexandre Tchernev} 
\address{Department of Mathematics\\
         University at Albany, SUNY\\ 
         Albany, NY 12222}
\email{atchernev@albany.edu}

\begin{abstract} 
We introduce  
notions of \emph{vector field} and its (discrete time) \emph{flow} on 
a chain complex. The resulting dynamical 
systems theory provides a set of tools with a broad range of 
applicability that allow, among others, to replace in a canonical 
way a chain complex with a 
``smaller'' one of the same homotopy type. As applications 
we construct in an explicit, canonical, and symmetry-preserving  
fashion a minimal free resolution for every toric ring and every 
monomial ideal.
Our constructions work in all characteristics and over any base field. 
A key subtle new point is that   
in certain finitely many positive characteristics (which depend on 
the object that is being resolved) a  
transcendental extension of the base field 
is produced before a resolution  
is obtained, while in all other characteristics 
the base field is kept unchanged. In the monomial case we show that 
such a transcendental base field extension \emph{cannot} in general 
be avoided, and we conjecture that the same holds in the toric case.  
\end{abstract} 

\maketitle

\section*{Introduction}

A standard exercise in differential geometry shows that a vector 
field on a smooth manifold induces naturally a chain homotopy $V$ on 
the de Rham complex of the manifold such that  $V^2=0$. In his 
work on discretizing Morse functions and vector fields on manifolds
\cite{Fo1, Fo2}, Forman arrives at the notions of  
\emph{combinatorial vector field} and its (discrete time) \emph{flow}
on the cellular chain 
complex of a regular CW-complex, and this combinatorial vector field 
is again a chain homotopy $V$ such that $V^2=0$. Chain homotopies with 
this property appear prominently also in homological perturbation 
theory, see~\cite{B, G, GL, L} and references there, and
implicitly in other 
related fields such as algebraic discrete Morse
theory~\cite{BW,JW,Sk,Sk2}. 

In this paper we 
present (a condensed version of) a general theory of dynamical
systems on chain complexes which in some sense mirrors the classical
theory of dynamical systems on compact manifolds, and which complements 
and incorporates as special instances the previously discussed notions.
Following Forman's lead, 
we call a chain homotopy $V$ on a chain complex $\cx F$ 
a \emph{vector field} on $\cx F$ whenever $V^2=0$. 
In analogy to the manifolds case, such a vector field induces a 
(discrete time) \emph{flow} $\Phi_V$ on $\cx F$, 
which is an endomorphism of $\cx F$ 
and is chain homotopic to the identity.
Iterating this flow produces a well-behaved dynamical 
system on $\cx F$ that preserves the homotopy type.  
To analyze its asymptotic behaviour, we employ a  
structure we call a $P$-\emph{grading} or $P$-\emph{stratification} 
on $\cx F$, the presence of which
is analogous to having a handlebody decomposition on a manifold, and
which is ubiquitous in applications. Such a stratification induces 
naturally a family of subquotients of $\cx F$ that we call the 
\emph{strata}. We use them to  
introduce the notion of \emph{Lyapunov structure} for a 
vector field $V$ on $\cx F$, which controls the asymptotic 
behaviour of the flow $\Phi_V$, and mirrors the way a 
Lyapunov function controls the asymptotic behaviour 
of a flow on a compact manifold, see e.g. \cite{C, Fr}. 
A key feature,  for which we have not yet 
discovered a good analogue in the classical dynamical systems on 
manifolds case, is that a $P$-stratification 
provides a set of tools for \emph{canonical} 
creation of vector fields on $\cx F$ with desired asymptotic flow 
behaviour.    

While a more detailed analysis of the resulting dynamical systems theory 
is an ongoing research project, even the basic facts 
presented here already have a broad range of 
applications. Leaving uses in topology and in 
representation theory for later publications, we focus on two 
long-standing problems from algebraic geometry, commutative algebra 
and combinatorics --- the construction in a canonical, explicit, 
and \emph{intrinsic} (i.e. symmetry-preserving) manner of minimal free 
resolutions for toric rings and for monomial ideals. 

The monomial ideals case has been a central open problem in commutative 
algebra since the thesis 
of Taylor~\cite{T} in 1966, and remains a very active area of research, see 
e.g.~\cite{MS, P, OW} and the references there.
The ultimate goal is a construction that is canonical, explicit, 
preserves the symmetries of the ideal, and works for all ideals and in 
all characteristics.  
A solution that works for all ideals 
in characteristic zero goes back at least to the 
main result of Yuzvinsky~\cite[Theorem 4.3]{Y}, and is obtained by 
using in it the canonical splittings given via  
Moore-Penrose pseudo-inverses, see Remark~\ref{R:pseudo-inverses}. 
In the 
recent paper~\cite{EMO}, Eagon, Miller, and Ordog use their amazing 
combinatorial description of Moore-Penrose pseudo-inverses    
to give another solution in characteristic zero with a strong         
combinatorial component, and this solution works also 
for every ideal in  
all but finitely many positive characteristics. 

To understand why a solution that works in 
all characteristics has remained elusive so far, we point out a common 
feature of every construction of a monomial resolution 
that we are aware of, whether minimal or not:  a base field is 
chosen at the start and is then kept fixed throughout regardless of the 
properties of the ideal that is being resolved. Perhaps unexpectedly, 
it turns out that such an approach \emph{cannot} succeed in positive 
characteristic if one aims for a  minimal free resolution 
that respects the symmetries of the monomial ideal. A key main result 
of this paper is Theorem~\ref{T:counterexample} where we provide for 
each prime $p\ge 2$ a monomial ideal $I(p)$ that does \emph{not} have 
a symmetry-preserving minimal free resolution for any base field  
algebraic over $\mathbb F_p$. 

The solutions we present in both the monomial and  the toric case employ 
a general strategy common to a number of  
constructions from representation theory, algebraic geometry, and 
commutative algebra, among them \cite{Y}, \cite{EMO}, and all constructions 
that use algebraic Morse theory.
In our dynamical systems terminology it is described as follows:
\begin{itemize}
\item   
start from a canonical 
resolution $\cx F$ with a suitable natural $P$-stratification; 

\item
then     
construct a splitting of each stratum; 

\item 
this induces in a natural way 
a vector field on $\cx F$ that has a flow stabilizing after finitely 
many iterations; 

\item
the stable iterate is then a projection of 
$\cx F$ onto a minimal resolution $\cx M$.
\end{itemize} 
Since each stratum is 
usually canonically obtained from an underlying finite chain complex of 
finite-dimensional vector spaces over the base field, 
with each vector space 
often given a canonical basis, this strategy reduces 
(completely formally, given our dynamical systems machinery) 
the minimal resolution problem to the problem of canonically 
constructing a splitting of such a chain complex of based vector spaces.   
We feature several ways of producing such a canonical splitting. 
   
Our construction in the monomial ideals case,  
Theorem~\ref{T:monomial-intrinsic}, obtains   
the explicit, canonical, and intrinsic minimal free resolution  
directly from the lcm-lattice~\cite{GPW} of the monomial ideal. 
It works for all monomial ideals
in all characteristics, and keeps the base field unchanged except 
for certain finitely many 
positive characteristics (which depend on the ideal), where it  
produces the desired result only after  
a transcendental extension of the base field. 
The starting canonical free resolution here is the \emph{lcm-resolution}, a 
non-minimal resolution supported on the order complex of the lcm-lattice. 
On each stratum there is a natural finite family of what we call 
\emph{matroidal} splittings, induced by the canonical bases of the vector 
spaces underlying the stratum. When the characteristic does not divide 
the number of matroidal splittings we obtain a canonical splitting of the stratum from  a simple average of all matroidal splittings. 
When the characteristic divides the number of matroidal splittings, we 
obtain a canonical splitting of the stratum only after a transcendental extension of the base field; here we form   
a weighted average of all matroidal splittings with generic weights. 

The problem for toric rings, while closely related to the problem 
for monomial ideals, has been open in all characteristics. 
Canonical, but non-minimal resolutions were constructed in 
\cite{BS} and \cite{TV1}. 
Our solution, Theorem~\ref{T:toric-intrinsic}, 
works in all characteristics and starts with 
the resolution from \cite{TV1}.  
To obtain a canonical splitting of each stratum  
in characteristic zero we use the Moore-Penrose splitting,  
and in positive characteristic we choose the  
weighted average with generic weights of \emph{all} 
splittings of the canonical finite chain complex 
that underlies the stratum. Because of this choice, in every positive 
characteristic we obtain a resolution only after a transcendental base 
field extension. Of course, see   
Remark~\ref{R:monomial-remarks}(g), 
this argument can easily be changed 
in a manner analogous to the monomial ideals case
so that the base field is kept unchanged except 
in finitely many positive characteristics. 

The paper is organized as follows.
Sections~\ref{S:vector-fields}-\ref{S:asymptotic-properties} 
contain a quick introduction to the general theory, with just 
the bare essentials needed for our applications. 
In Section~\ref{S:vector-fields} we introduce the notions of vector 
field on a chain complex, and the flow of a vector field. We also 
introduce the notions of $P$-grading/stratification, stratum, 
Lyapunov structure, and chain-recurrent chain complex. A key 
point is Proposition~\ref{P:stratification-induced}(c) which  
shows that, given a vector 
field on each stratum of a stratified complex $\cx F$, one obtains 
in an explicit canonical way a vector field on $\cx F$. 
In Section~\ref{S:vector-fields-and-splittings} we review
from our new perspective a well-known 
and well-studied class of vector fields, the class of all splittings.   
As their flows are already stable, they serve as an important building 
block to create vector fields with well controlled asymptotic behaviour. 
We give examples, and elaborate on how to obtain a new splitting  
from a weighted average of a finite set of splittings. 
In Section~\ref{S:asymptotic-properties} we show, see
Theorem~\ref{T:asymptotic}, how a Lyapunov 
structure can be used to control the asymptotic behaviour of 
the flow of a vector field. The proof we give is an elaborate and 
explicit generalization of an argument of Yuzvinsky 
from~\cite[Lemma~4.1]{Y}, but parts of our theorem can be obtained 
also by using the proof of the basic perturbation lemma from 
homological perturbation theory \cite{B,G}. 

In Section~\ref{S:toric-rings} we give our construction,
Theorem~\ref{T:toric-intrinsic}, 
of the explicit, canonical, and intrinsic minimal free 
resolution of a toric ring. 
We discuss some of its properties, and compute a simple example. 

In Section~\ref{S:monomial-resolutions} we present,
see Theorem~\ref{T:monomial-intrinsic}, our 
construction of an explicit,  
canonical, and intrinsic minimal free resolution of a monomial ideal.    
We state in Problem~\ref{P:mfr} a much weaker than normally considered 
version of this construction problem, which reflects the necessity 
of passing to a base field extension before a solution can be obtained
in some positive characteristics. 
Using our dynamical systems approach we give a solution that  
has good functorial properties and works in all characteristics. 
Section~\ref{S:an-example} is devoted to the computation of a 
useful example of this monomial construction. 

Finally, in
Section~\ref{S:monomial-resolutions-in-positive-characteristic}
we show that, see Theorem~\ref{T:counterexample}, 
in positive characteristic it is not, in general, possible to construct 
a minimal resolution of a monomial ideal by using only intrinsic 
properties of the ideal without first taking a transcendental 
extension of the base field. 

The author would like to thank Lucho Avramov, Hank Kurland, and 
especially Marco Varisco for very useful conversations on 
various aspects of the theory presented here.

\section{Vector fields}\label{S:vector-fields}

Throughout this paper rings are associative with unit, modules 
are left and unitary, and unadorned tensor products are 
over $\mathbb Z$. 
While all definitions and 
results in the first three sections will be formulated in 
the setting of chain complexes of modules over some base ring, 
it should be clear to the reader that 
they are valid also for chain complexes 
over any co-complete abelian category. 

Let $\cx F=(F_n,\phi_n)$ be a chain complex 
of modules over a ring.
We will write $B_n(\cx F)$, \  
$Z_n(\cx F)$, and $\HH_n(\cx F)$ 
for the $n$-boundaries, $n$-cycles, and $n$th homology,   
respectively, of $\cx F$.  
The \emph{homology chain complex of\/ $\cx F$} is 
the chain complex $\HH(\cx F)$ with zero 
differential and component $\HH_n(\cx F)$ in 
homological degree $n$.

Recall that  
if $D$ is a chain homotopy on $\cx F$ then the morphism 
\[
\Phi_D=\id_{\cx F} - \phi D - D \phi
\] 
is called the 
\emph{deformation of $\cx F$ along $D$}. 
Clearly $\Phi_D$ is an endomorphism of $\cx F$ and 
is chain homotopic to the identity. 

\begin{definition}
A \emph{pre-vector field} on $\cx F$ is a  
chain homotopy $D$ on $\cx F$ such that 
\[
D\Phi_D=\Phi_D D.
\] 
A \emph{vector field} on $\cx F$ is a chain homotopy $V$ 
such that $V^2=0$. 
If this is the case, we call the deformation 
$\Phi_V$ the \emph{flow} of the vector field $V$. 
\end{definition}

\begin{example}
Recall that a \emph{contraction} of $\cx F$ is a 
chain homotopy $D$ such that $\Phi_D=0$. Thus  
every contraction is a pre-vector field on $\cx F$. 
Furthermore, a routine computation shows that 
if $D$ is a contraction, then the 
homotopy $D\phi D$ is again a contraction and 
also a vector field on $\cx F$. 
\end{example}

The following proposition is a straighforward consequence
of the definitions.

\begin{proposition} 
Let $D$ be a chain homotopy on $\cx F$.  

(a) $D$ is a pre-vector field on $\cx F$ if and only if 
$D^2\phi = \phi D^2$. In particular, 
a vector field on $\cx F$ is also a pre-vector field. 

(b) If $D$ is a pre-vector field on $\cx F$, then it is 
also a pre-vector field on both $\Ker\Phi_D$ and $\im\Phi_D$, 
and a contraction on $\Ker\Phi_D$. 
In particular, $\im\Phi_D$ is chain homotopy 
equivalent to $\cx F$.    

(c) If $D$ is a pre-vector field on $\cx F$, then the 
chain homotopy $V=D\phi D$ is a vector field on $\cx F$.  
Furthermore, $V$ is a vector field 
on both $\Ker\Phi_D$ and $\im\Phi_D$, and a contraction 
on $\Ker\Phi_D$.   
\end{proposition}

Thus, iterating the flow of a vector field on a chain complex
produces a well-behaved dynamical system that preserves homotopy type. The
structure described below emerges as a key tool in controling
the asymptotic behaviour of such a dynamical system. 
It can be thought of as an analogue in the setting of chain 
complexes to a handlebody decomposition 
of a compact smooth manifold.

\begin{definition}
Let $P$ be a poset. 
A \emph{$P$-grading} or \emph{$P$-stratification}
on the  complex $\cx F$ is a decomposition  
\[
F_n=\bigoplus_{a\in P} F_n^a 
\] 
for every $n$ 
such that for each $a\in P$ the 
collection $\cx F(a)=(F_n(a), \phi_n)$ 
is a subcomplex of $\cx F$, where 
\[
F_n(a)=\bigoplus_{x\le a}F_n^x. 
\]
The summand $F_n^a$ is called the 
\emph{degree $a$ homogeneous component} 
of $F_n$. The chain complex
\[
\ol{\cx F(a)} = \cx F(a)\Big/\sum_{b<a}\cx F(b)
\]
is called the \emph{associated graded strand} at $a$ or the
\emph{stratum} at $a$. Observe that we have
$\ol{\cx F(a)}_n = F_n^a$. 
\end{definition}

$P$-gradings are ubiquitous in algebraic geometry,  
commutative algebra and algebraic topology. Below we provide several
relevant examples. For more details, examples, and  
applications see e.g. \cite[Section 2]{CT} and the references there.

\begin{example}
(a)   
Let $P=pt$, and set  
$F_n^{pt}=F_n$. This produces the \emph{trivial} 
grading/stratification on $\cx F$.

(b)
Let $P=\mathbb Z$ and set $F_n^n = F_n$ and $F_n^a = 0$ if $a\ne n$.
This produces the $\mathbb Z$-grading/stratification 
by \emph{homological degree}. 
\end{example}

\begin{example} 
Let $R$ be a $\mathbb Z$-graded ring, and let $\cx F$ be a complex 
of $\mathbb Z$-graded free $R$-modules, with degree-preserving 
differentials. For each $n$ let $B_n$ be a homogeneous basis of $F_n$, 
and for each $a$ let $B_n^a$ be the subset of all elements $b\in B_n$ of 
degree $|b|=a$. Writing $F_n^a$ for the free submodule of $F_n$ with 
basis the elements from $B_n^a$, we obtain a $\mathbb Z$-stratification 
of $\cx F$. Note that while this stratification depends on the 
chosen bases $B_n$, the complexes $\cx F(a)$ and the strata $\ol{\cx F(a)}$
do not depend on this choice.   
\end{example}

\begin{example}\label{Ex:double-complex}
Let $C_{\bullet\bullet}$ be a double complex and let 
$\cx F=Tot^\oplus(C_{\bullet\bullet})$ be the associated total chain complex, 
see e.g.~\cite[Section~1.2]{W}. Thus $F_n=\bigoplus_{i+j=n}C_{ij}$ and the 
two canonical filtrations of $\cx F$ induce two canonical 
$\mathbb Z$-stratifications on $\cx F$. For example, setting 
$F_n^a=C_{a,n-a}$ gives the \emph{stratification by columns} of $\cx F$, 
and the stratum at $a$ is $\ol{\cx F(a)}=C_{a,\bullet}$ with  
differential the $a$th vertical differential of $C_{\bullet\bullet}$. 
\end{example}

\begin{example} 
Let $X$ be a regular CW-complex, and let $\Bbbk$ be a 
commutative ring. Let $C_\bullet(X,\Bbbk)$ be 
the cellular chain complex of $X$ with coefficients in $\Bbbk$. 
Recall that the set $A_n$ of 
the $n$-cells $\sigma$ of $X$ (with closure $\bar\sigma$ and boundary 
$\dot\sigma$) index a decomposition of 
\[
F_n = \HH_n(X^n,X^{n-1},\Bbbk) = 
\bigoplus_{\sigma\in A_n}\HH_n(\bar\sigma,\dot\sigma, \Bbbk).  
\]
The face poset $P(X)=\bigcup_{n\ge 0}A_n$ of $X$ has partial order given 
by $\sigma\le \tau$ if and only if $\sigma\subseteq\bar\tau$. 
Thus we obtain the standard $P(X)$-stratification of $C_\bullet(X,\Bbbk)$
where 
$F_n^\sigma = \HH_n(\bar\sigma, \dot\sigma ,\Bbbk)$ if $\sigma\in A_n$ 
and $F_n^\sigma=0$ otherwise. 
\end{example}

\begin{example}[Taylor resolution]\label{Ex:Taylor-1} 
Let $U$ be a finite set, 
let $\Delta^U$ be the geometric realization of the 
full simplex on the set of vertices $U$. Thus $\Delta^U$ is 
a regular CW-complex and the 
$n$-cells $[\sigma]$ of $\Delta^U$ are indexed by the 
elements $\sigma$ of the set $A_n$ of $n+1$-subsets of 
$U$. Let $X$ be a set of variables and let 
$M(X)=\{f \mid f\: X\lra \mathbb N\}$ be the free abelian 
monoid on the set $X$ with operation value-wise addition. 
Let $\Bbbk$ be a commutative ring. We 
identify each element $f\in M(X)$ with the monomial 
$\prod_{y\in X}y^{f(y)}$ in the polynomial ring $\Bbbk[X]$ over 
$\Bbbk$ on the set of variables $X$. For each $v\in U$ let 
$m_v$ be a monomial in $\Bbbk[X]$ such that $m_u$ does not 
divide $m_v$ when $u\ne v$. More generally, with every subset 
$\sigma$ of $U$ we associate the monomial 
$m_\sigma=\lcm\{m_v\mid v\in\sigma\}$. The set of monomials 
$L=\{ m_\sigma \mid \sigma\subseteq U\}$, partially ordered 
by divisibility, is in fact a lattice called the 
\emph{lcm-lattice} of the ideal $I$ generated in $\Bbbk [X]$ 
by the set of monomials $\{m_u\mid u\in U\}$, see \cite{GPW}.  
Let $d_n^{\sigma\tau}$ be
the component of the differential of the cellular 
chain complex $C_\bullet(\Delta^U, \mathbb Z)$ that sends 
$\HH_n\bigl(\ol{[\sigma]},\dot{[\sigma]}, \mathbb Z\bigr)$ to 
$\HH_{n-1}\bigl(\ol{[\tau]},\dot{[\tau]}, \mathbb Z\bigr)$. 
The \emph{Taylor resolution} \cite{T} of the ideal $I$ 
is the chain complex $\cx T=\cx T(I)=(T_n,\delta_n)$ given by 
$
T_n=
\bigoplus_{\sigma\in A_n}
\HH_n\bigl(\ol{[\sigma]}, \dot{[\sigma]},\mathbb Z\bigr)\otimes\Bbbk[X]
$ 
with differential $\delta_n$ given by 
\[
\delta_n\Big\vert_{\HH_n\bigl(\ol{[\sigma]}, \dot{[\sigma]},\mathbb Z\bigr)\otimes\Bbbk[X]} 
= 
\sum_{\tau\in A_{n-1}}d_n^{\sigma\tau}\otimes (m_\sigma/m_\tau). 
\] 
Now by setting 
$
T_n^a = 
\bigoplus_{m_\sigma=a}
\HH_n\bigl(\ol{[\sigma]}, \dot{[\sigma]},\mathbb Z\bigr)\otimes\Bbbk[X] 
$
for every $a\in L$ we obtain an $L$-stratification on $\cx T$. 
For $a\in L$ let $\Delta_{\le a}$ (respectively, $\Delta_{<a}$) 
be the union of all cells $[\sigma]$ of $\Delta^U$ such that 
$m_\sigma\le a$ (respectively, $m_\sigma<a$). It is straightforward 
to check $\Delta_{\le a}$ and $\Delta_{<a}$ are regular CW-subcomplexes 
of $\Delta^U$, and that the stratum $\ol{\cx T(a)}$ is 
precisely 
\[
\ol{\cx T(a)} = 
C_\bullet(\Delta_{\le a},\Delta_{<a},\mathbb Z)\otimes\Bbbk[X],
\] 
where 
$
C_\bullet(\Delta_{\le a},\Delta_{<a},\mathbb Z)=
C_\bullet(\Delta_{\le a},\mathbb Z)/ 
C_\bullet(\Delta_{<a},\mathbb Z)
$
is the cellular chain complex of the pair $(\Delta_{\le a},\Delta_{<a})$.  
\end{example}

\begin{definition}
Let $P$ be a poset.
  
(a)
Let $V$ be a vector field on $\cx F$. 
We say that $V$ is \emph{compatible} with a $P$-grading
on $\cx F$ if 
for each $a\in P$ and each $n$ the homotopy $V$ 
maps the degree $a$ homogeneous component 
$F_n^a$ into the homogeneous component $F_{n+1}^a$.

(b) 
A \emph{Lyapunov} (or \emph{chain recurrent}) 
\emph{structure 
for $V$ on $\cx F$ with values in $P$} is a 
a $P$-stratification on $\cx F$ together with a 
compatible vector field $V$.

(c) 
If given a 
Lyapunov structure with values in $P$ for a vector field $V$ 
on $\cx F$, then   
for each $a\in P$ the stratum  
\[
\ol{\cx F(a)}=\cx F(a)\Big/\sum_{b<a}\cx F(b)
\]  
is also called the \emph{chain recurrent} chain 
complex of $V$ at $a$.  
\end{definition}

The following basic properties are immediate from the 
definitions.  

\begin{proposition}\label{P:stratification-induced} 
Fix a $P$-stratification on $\cx F$. 

(a) 
A compatible vector field $V$ 
induces for each $a\in P$ 
a canonical vector field $\ol{V(a)}$ on the 
stratum $\ol{\cx F(a)}$. 

(b) 
Suppose that 
for  each $a\in P$ we are given a vector field \ 
$W_a$ on the chain complex $\ol{\cx F(a)}$ . 
Via the inclusion $\ol{\cx F(a)}_n=F_n^a\subseteq F_n$,  
each $W_a$ induces a vector field (also denoted by $W_a$) 
on $\cx F$ by setting $W_a(F_n^b)=0$ for $b\ne a$. 

(c)
Let the vector fields $W_a$ be as in part (b), and 
let $W=\sum_{a\in P}W_a$.  
Then $W$ is a vector field on $\cx F$ 
compatible with the given stratification, 
and such that $\ol{W(a)}=W_a$ for every $a\in P$. 
\end{proposition}

\begin{remark}
For the applications we consider in this paper, one starts with a chain
complex $\cx F$ that has a natural $P$-stratification, but
is ``too big''. In that situation one way to find a smaller complex
of the same homotopy type is to construct a vector field $W$
on $\cx F$ compatible with the stratification and such that 
the dynamical system produced by iterating the flow $\Phi_W$
stabilizes in any given homological degree after finitely many 
iterations. Then one can
take the stable image as the desired smaller chain complex.
Proposition~\ref{P:stratification-induced}(c), while an
elementary consequence of the definitions, plays a crucial
role in this approach. Note also that, while most of the
theory developed so far works also for pre-vector fields, part (c)
of this proposition is where we really require vector fields.  
\end{remark}

\begin{example}\label{Ex:Taylor-3} 
With notation as in Example~\ref{Ex:Taylor-1}, for each $a\in L$ 
let $W_a$ be a vector field on 
$
C_\bullet(\Delta_{\le a},\Delta_{<a},\Bbbk) = 
C_\bullet(\Delta_{\le a},\Delta_{<a},\mathbb Z)\otimes\Bbbk
$. 
Then $W_a\otimes 1$ is a vector field on the 
chain complex 
$
\ol{\cx T(a)} = 
C_\bullet(\Delta_{\le a},\Delta_{<a},\Bbbk)\otimes_\Bbbk\Bbbk[X]
$, 
hence by Proposition~\ref{P:stratification-induced} we obtain a vector field 
$W=\sum_{a\in L} W_a\otimes 1$ on the Taylor resolution $\cx T$ 
compatible with the $L$-stratification.  
\end{example}

\section{Vector fields and splittings}
\label{S:vector-fields-and-splittings}

Here we briefly review a well-known class of vector fields whose 
flows are projections, hence already stable. 
They will be used later as building blocks to construct 
vector fields with desired asymptotic behaviour.   

Recall that the chain complex $\cx F$ \emph{splits} or 
\emph{is splittable} if it is chain homotopy equivalent 
to its homology complex $\HH(\cx F)$. Clearly this 
happens if and only if there are morphisms of complexes 
$f\:\cx F\ra\HH(\cx F)$ and $g\:\HH(\cx F)\ra\cx F$ 
and a chain homotopy $D$ on $\cx F$ such that 
$fg=\id_{\HH(\cx F)}$, and  $gf=\Phi_D$ (we say that $D$ 
\emph{realizes} a splitting of $\cx F$).  It is 
a standard calculation to verify that 
this is also equivalent to $\cx F$ having a
vector field $D$ such  that 
\begin{equation}\label{E:splitting}
\phi D\phi = \phi.   
\end{equation} 
and  
\begin{equation}\label{E:partial-splitting}
D\phi D = D. 
\end{equation}
Notice that if a vector field $D$ on $\cx F$ 
satisfies \eqref{E:partial-splitting} 
then $\Phi_D$ is an idempotent on $\cx F$, hence   
$\im\Phi_D$ is homotopic to, and a direct summand of, $\cx F$. 
However, unless an additional condition 
like \eqref{E:splitting} holds, $\im\Phi_D$ will 
not be isomorphic to $\HH(\cx F)$.  
This motivates the following terminology.

\begin{definition}\label{D:splittings} 
A \emph{partial splitting} of $\cx F$ is a vector 
field $D$ on $\cx F$ that satisfies 
\eqref{E:partial-splitting}.  
A \emph{splitting} of $\cx F$ is a partial splitting 
that satisfies \eqref{E:splitting}. 
\end{definition}

Let $D$ be a chain homotopy on $\cx F$. 
Set $N_{n+1}=\im(\phi D|_{F_n})$, let   
$M_n=\im (D\phi|_{F_n})$, and let 
$C_n=\Ker (D\phi|_{F_n})\cap\Ker(\phi D|_{F_n})$. 
If $D$ satisfies \eqref{E:splitting} or 
\eqref{E:partial-splitting} then 
both $D\phi$ and $\phi D$    
are idempotents on $\cx F$, and  
$F_n$ decomposes for each $n$ as 
\begin{equation}\label{E:near-splitting}
F_n=N_{n+1}\oplus C_n\oplus M_n. 
\end{equation} 
Since such a decomposition occurs also in  
other important cases, we are prompted  
to make the following more general definition. 

\begin{definition}
We say that a chain homotopy $D$ is a
\emph{weak partial splitting} of $\cx F$ if 
for each $n$ the decomposition \eqref{E:near-splitting} 
holds, and $\phi D$ and $D\phi$ induce automorphisms 
on $N_{n+1}$ and $M_n$, respectively, for each $n$. 
\end{definition}

\begin{example}\label{Ex:adjoint}
Suppose that $\Bbbk$ is a subfield of $\mathbb C$, 
and that $\cx F$ is a chain
complex of finite dimensional vector spaces over $\Bbbk$ such that 
each $F_i$ is given an inner product. For each $i$
let $\phi_i^*$ be the adjoint to $\phi_{i}$. Thus 
the chain homotopy
$\phi^*$ is a vector field on $\cx F$, and it is a standard fact
from basic linear algebra that
$\phi^*$ is a weak partial splitting of $\cx F$. 
\end{example}

\begin{proposition}\label{P:weak-partial-splitting} 
Let $D$ be a weak partial splitting of $\cx F$. 

(a) 
For each $n$ the maps $D$ and $\phi$ induce 
isomorphisms between $N_n$ and $M_n$. 

(b) 
Define the chain homotopy $\widehat D$ on $\cx F$ 
by $D(\phi D)^{-1}$ on $N_{n+1}$  
and by $0$ on $M_n$ and $C_n$, for each $n$. 
Then $\widehat D$ is 
a partial splitting on $\cx F$. 

(c) 
If $D$ satisfies \eqref{E:splitting} or \eqref{E:partial-splitting} then 
$\widehat D = D\phi D(I - D\phi)$.

(d)
If $D$ satisfies \eqref{E:splitting} then 
$\widehat D$ is a splitting.

(e)
If $D$ is a partial splitting then $\widehat D = D$. 
\end{proposition}

\begin{proof} 
Note that $D(N_n)=D\phi D(N_n)\subseteq M_n$
hence $D\: N_n \lra M_n$ is a monomorphism. Furthermore, 
$\phi(M_n)=\phi D\phi(M_n)\subseteq N_n$, hence 
$\phi\: M_n\lra N_n$ is also a monomorphism. Since 
$\phi(D(N_n))=N_n$ we see that $\phi\: M_n\lra N_n$ is 
also epi hence isomorphism. Finally, since 
$D(\phi(M_n))=M_n$ we see  that $D\: N_n\lra M_n$ is 
also epi hence isomorphism. This takes care of part (a). 

(b) 
Since $\widehat D(F_n)=D(N_{n+1})=M_{n+1}$ by part (a), 
it is clear that $\widehat D^2=0$. Furthermore,  
$\widehat D\phi\widehat D=\widehat D=0$ 
on $C_n\oplus M_n$ for each $n$. Finally, on $N_{n+1}$ we have 
$\widehat D\phi\widehat D=\widehat D\phi D(\phi D)^{-1} =
\widehat D$. 

(c)
Since $\phi D$ is the identity on $N_{n+1}$, we have that
$\widehat D$ is given by $D$ on $N_{n+1}$ and by $0$ on $C_n\oplus M_n$.
As $D\phi D(I- D\phi)$ satisfies the same condition, it equals $\widehat D$. 

(d) and (e) are 
immediate consequences of (b) and (c).  
\end{proof}

\begin{remark}\label{R:affine-combination}
Suppose $D_1, \dots, D_n$ are homotopies of $\cx F$ that satsify
\eqref{E:splitting}, 
and $r_1,\dots, r_n$ are central elements of our base
ring such that $r_1+\dots +r_n=1$.
Then $D=r_1D_1 + \dots + r_nD_n$ satisfies \eqref{E:splitting}, 
and therefore $\widehat D$ is a
splitting by Proposition~\ref{P:weak-partial-splitting}. 
\end{remark}

\begin{remark}[Moore-Penrose splitting]\label{R:pseudo-inverses}
Suppose that $\Bbbk$ is a subfield of $\mathbb C$, 
and that $\cx F$ is a chain
complex of finite dimensional vector spaces over $\Bbbk$ such that 
each $F_i$ is given an inner product. For each $i$
let $\phi_i^*$ be the adjoint to $\phi_{i}$.
As observed in Example~\ref{Ex:adjoint},  
the chain homotopy $\phi^*$ is a
weak partial splitting of $\cx F$. The  partial
splitting  $\phi^{+}=\widehat{\phi^*}$ in this case is given
by taking the well known Moore-Penrose \cite{M, P, Bj} 
pseudo-inverses $\phi_i^{+}$ 
of the maps $\phi_{i}$; in particular, $\phi^+$ is a splitting.  
Indeed, let $f_i(x)=\det(xI - \phi_{i}\phi_i^*)$ be the characteristic 
polynomial of $\phi_{i}\phi_i^*$. Since $\phi_{i}\phi_i^*$ is 
self-adjoint, we have $f_i(x)=x^{n_i}g_i(x)$ where
$n_i = \dim F_{i-1} - \rank \phi_{i}$, and $g_i(0)\ne 0$. It follows that 
\[
1 - \frac{g_i(x)}{g_i(0)} = x p_i(x)
\]
for a canonically and explicitly determined by $\phi_{i}$ and $\phi_i^*$ 
polynomial $p_i(x)$. It is
now standard to check that  $p_i(\phi_{i}\phi_i^*)$ is the inverse of
$\phi_{i}\phi_i^*$ on  $N_{i}=\im(\phi_{i}\phi_i^*)$, and that
$\phi_{i}\phi_i^*p_i(\phi_{i}\phi_i^*)$ is the orthogonal projection
of $F_{i-1}$ onto $N_{i}$. Therefore, 
by Proposition~\ref{P:weak-partial-splitting}, we obtain the 
explicit formula
\begin{equation}\label{E:pseudo-inverse} 
\phi_i^{+} =
\phi_i^* p_i(\phi_{i}\phi_i^*)\phi_{i}\phi_i^* p_i(\phi_{i}\phi_i^*).  
\end{equation}
It is trivial to check from this formula that $\phi_i^{+}$ satisfies  
the Moore-Penrose conditions \cite[p. 290]{GVL} hence equals the 
Moore-Penrose pseudo-inverse of $\phi_{i}$; in particular,
we recovered its standard canonical explicit 
polynomial expression in terms of $\phi_{i}$ and its adjoint. 
\end{remark}

\begin{example}\label{Ex:pseudo-inverses} 
Consider the chain complex of $\mathbb Q$-vector spaces 
\[
0 \lla 
\mathbb Q 
\xleftarrow[\phi_1]{
[
\begin{smallmatrix}
1 & 1 & 1 & 1 & 1 & 1 & 1 
\end{smallmatrix}
]
} 
\mathbb Q^{7} 
\xleftarrow[\phi_2]{
\left[
\begin{smallmatrix}
\hm 0 & \hm 0 & \hm 0 & \hm 0 & \hm 0 & \hm 0  \\ 
-1    &    -1 & \hm 0 & \hm 0 & \hm 0 & \hm 0  \\ 
\hm 0 & \hm 0 &    -1 &    -1 & \hm 0 & \hm 0  \\ 
\hm 0 & \hm 0 & \hm 0 & \hm 0 &    -1 &    -1  \\ 
\hm 1 & \hm 0 & \hm 1 & \hm 0 & \hm 0 & \hm 0  \\ 
\hm 0 & \hm 1 & \hm 0 & \hm 0 & \hm 1 & \hm 0  \\
\hm 0 & \hm 0 & \hm 0 & \hm 1 & \hm 0 & \hm 1
\end{smallmatrix}
\right]
}
\mathbb Q^{6} 
\lla 0, 
\]
where the inner products are the ones making the standard  
basis on each $\mathbb Q^m$ orthonormal.  
We have 
\[
\phi_2\phi_2^* = 
\left[
\begin{smallmatrix}
 0 & \hm 0 & \hm 0 & \hm 0 & \hm 0 & \hm 0 & \hm 0  \\
 0 & \hm 2 & \hm 0 & \hm 0 &    -1 &    -1 & \hm 0  \\
 0 & \hm 0 & \hm 2 & \hm 0 &    -1 & \hm 0 &    -1  \\
 0 & \hm 0 & \hm 0 & \hm 2 & \hm 0 &    -1 &    -1  \\
 0 &    -1 &    -1 & \hm 0 & \hm 2 & \hm 0 & \hm 0  \\
 0 &    -1 & \hm 0 &    -1 & \hm 0 & \hm 2 & \hm 0  \\ 
 0 & \hm 0 &    -1 &    -1 & \hm 0 & \hm 0 & \hm 2  
\end{smallmatrix}
\right]
\]
which has 
characteristic polynomial 
$
f_2(x) = \det(xI - \phi_2\phi^*_2) = 
x^7  - 12x^6  + 54x^5  - 112x^4  + 105x^3  - 36x^2
$, 
therefore 
\[
p_2(x) = 
\frac{1}{36}(x^4 - 12x^3  + 54x^2 - 112x + 105),
\]  
and thus we obtain by \eqref{E:pseudo-inverse} that 
\[
\phi_2^+ = 
\frac{1}{12}
\left[ 
\begin{smallmatrix}
      0& {-5}&\hm 3& {-1}&\hm 5& {-3}&\hm 1\\
      0& {-5}& {-1}&\hm 3& {-3}&\hm 5&\hm 1\\
      0&\hm 3& {-5}& {-1}&\hm 5&\hm 1& {-3}\\
      0& {-1}& {-5}&\hm 3& {-3}&\hm 1&\hm 5\\
      0&\hm 3& {-1}& {-5}&\hm 1&\hm 5& {-3}\\
      0& {-1}&\hm 3& {-5}&\hm 1& {-3}&\hm 5
\end{smallmatrix}
\right]. 
\]
Doing a similar (but much simpler) computation for $\phi^+_1$ 
yields that the Moore-Penrose splitting $\phi^+$ has the form 
\[
0\lra \mathbb Q 
\xrightarrow[\phi_1^+]{
\left[ 
\begin{smallmatrix} 
1/7 \\
1/7 \\
1/7 \\
1/7 \\
1/7 \\
1/7 \\ 
1/7
\end{smallmatrix}
\right] 
}
\mathbb Q^7 
\xrightarrow[\phi_2^+]{
1/12 
\left[
\begin{smallmatrix} 
      0& {-5}&\hm 3& {-1}&\hm 5& {-3}&\hm 1\\
      0& {-5}& {-1}&\hm 3& {-3}&\hm 5&\hm 1\\
      0&\hm 3& {-5}& {-1}&\hm 5&\hm 1& {-3}\\
      0& {-1}& {-5}&\hm 3& {-3}&\hm 1&\hm 5\\
      0&\hm 3& {-1}& {-5}&\hm 1&\hm 5& {-3}\\
      0& {-1}&\hm 3& {-5}&\hm 1& {-3}&\hm 5
\end{smallmatrix}
\right]
}
\mathbb Q^6 \lra 0.  
\]
\end{example}

\begin{example}\label{Ex:Taylor-0}
With notation as in Example~\ref{Ex:Taylor-3}, take $\Bbbk =\mathbb Q$. 
For each $n$ we have 
$
C_n(\Delta_{\le a},\Delta_{<a},\mathbb Z) = 
\bigoplus_{m_\sigma=a}\HH_n\bigl(\ol{[\sigma]},\dot{[\sigma]},\mathbb Z\bigr)
$. 
Since each $\HH_n\bigl(\ol{[\sigma]},\dot{[\sigma]},\mathbb Z\bigr)$ is 
just a copy of $\mathbb Z$ it has a unique up to sign free generator. 
This yields a natural up to signs basis of
$C_n=C_n(\Delta_{\le a},\Delta_{<a},\mathbb Q)$ and therefore induces a natural 
independent of the signs inner product on $C_n$ in which this basis is 
orthonormal. Thus, as in Remark~\ref{R:pseudo-inverses} we obtain an 
explicit natural splitting $d_a^{+}$ of 
$C_\bullet(\Delta_{\le a},\Delta_{<a},\mathbb Q)$ by taking the Moore-Penrose 
pseudo-inverses of its differentials. This induces an explicit natural 
splitting $\delta_a^{+}=d_a^+\otimes 1$ on the stratum $\ol{\cx T(a)}$, 
hence, by Proposition~\ref{P:stratification-induced}, 
an explicit natural vector field $\delta^{+}=\sum\delta_a^{+}$ on 
the Taylor resolution $\cx T$ compatible with the $L$-stratification.     
\end{example}

\section{Asymptotic properties}
\label{S:asymptotic-properties}

Just as in the classical dynamical systems case, where the
asymptotic properties of the flow of 
a vector field on a smooth compact manifold
can be controlled by a Lyapunov function, see e.g. \cite{C, Fr},
the behaviour in the long run of the flow of a vector 
field on a chain complex can be controlled by a Lyapunov 
structure. 

To understand this mechanism, 
recall that if $P$ is a poset and $a\in P$ then the 
\emph{dimension} of $a$ in $P$ is the supremum $d_P(a)$ 
of the lengths $n$ of strictly increasing chains 
$a_0<\dots < a_n=a$ in $P$ terminating at $a$. 
The \emph{dimension} $\dim P$ of $P$ is the supremum 
of the dimensions $d_P(a)$ of all the elements $a$ of $P$ . 
The following has been a main motivating example for us. 

\begin{example}\label{Ex:Taylor-4}  
With notation as in Example~\ref{Ex:Taylor-3}, suppose that 
$W_a$ is a splitting of $\ol{\cx T(a)}$ for each $a\in L$. 
Then, although not explicity included in its statement, 
the proof of \cite[Lemma~4.1]{Y} also shows that 
$\Phi_W^{k+1}=\Phi_W^k$ whenever $k>\dim L$. Thus, $\dim L$ 
controls how quickly the flow of $W$ stabilizes. Furthermore, 
this implies that $\Phi_W^{1+\dim L}$ is a projection of $\cx T$ onto 
a chain homotopic direct summand $\cx M_W$ of $\cx T$. 
\end{example}

Another relevant instance of this mechanism is the proof of  
the basic perturbation lemma from homological perturbation theory, see 
e.g.~\cite[Basic Perturbation Lemma 2.4.1]{GL}, its argument 
going back to \cite{G,B,S}.  
It can be understood in our language as giving a sufficient
condition that certain vector fields have flows that stabilize
in any given homological degree after finitely may iterations.
As Example~\ref{Ex:double-complex} suggests, a similar 
interpretation can be given also to Eagon's construction 
of a Wall complex~\cite[Theorem~1.2]{E}; 
see also \cite[Lemma~3.5]{EFS} for a categorical version of
this construction.

The next result, which is the main result in this section, 
can be thought of as an elaborate and explicit
generalization of Yuzvinsky's argument we referred to 
in Example~\ref{Ex:Taylor-4}. 
It can also be interpreted as a detailed analysis,
from our dynamical systems viewpoint, of
a special case of the basic perturbation lemma of homological perturbation
theory. Either way, this is the key technical fact needed for 
the applications we consider in this paper.

\begin{theorem}\label{T:asymptotic}
Fix a $P$-grading on $\cx F$. For each $c\in P$ let $V_c$ 
be a vector field on the chain complex $\ol{\cx F(c)}$, and 
let $V=\oplus_c V_c$ be the induced vector field on $\cx F$. Consider 
the subposet  $P_n=\{c\in P\mid F_n^c\ne 0\}$ of $P$, 
let $a\in P_n$ be of finite dimension, 
suppose that for each $c\in P_n$ with $c\le a$ the vector field 
$V_c$ is a partial splitting on $\ol{\cx F(c)}$, and     
let $F_n^c=N^c_{n+1}\oplus C^c_n\oplus M^c_n$ be the corresponding 
decomposition as in \eqref{E:near-splitting}. 

Then for any 
collection of nonnegative integers $\{k_c\mid c\le a\}$ 
we have 
\[
\sum_{c\le a}\Phi_V^{k_c}(C^c_n)=
\bigoplus_{c\le a}\Phi_V^{k_c}(C^c_n),
\] 
and each restriction $\Phi_V^{k_c}\: C^c_n \lra \Phi_V^{k_c}(C^c_n)$ 
is an isomorphism.  Furthermore, 
for each integer $k > d_{P_n}(a)$ we also have:   
\begin{enumerate}
\item 
$\Phi_V^k(M_n^a)=0$; 

\item 
$[\Phi_V -\id](C_n^a)\subseteq\bigoplus_{b<a}M_n^b$; 

\item
$[\Phi_V^{k}-\Phi_V^{k-1}](C_n^a)=0$; 

\item
$
\Phi_V^k(N_{n+1}^a)\subseteq\bigoplus_{b<a}\Phi_V^{d_{P_n}(b)}(C_n^b);
$ 

\item 
$[\Phi_V^{k+1}-\Phi_V^k](F_n^a)=0$; 

\item
$
\Phi_V^k(F_n^a)\subseteq
\bigoplus_{b\le a}\Phi_V^{d_{P_n}(b)}(C_n^b).
$

\item 
$[\Phi_V^{k+1}-\Phi_V^k]\bigl(\cx F(a)_n\bigr)=0$. 

\item 
$
\Phi_V^k\bigl(\cx F(a)_n\bigr)=
\bigoplus_{b\le a}\Phi_V^{d_{P_n}(b)}(C_n^b).
$ 
\end{enumerate}  
\end{theorem}

\begin{proof}
We use induction on $t=d_{P_n}(a)$. When $t=0$ then $a$ is 
a minimal element of $P_n$, hence $\cx F(a)_n=\ol{\cx F(a)}_n$,  
therefore $V=V_a$ on $\cx F(a)_n$ and  
$\Phi_{V_a}=\Phi_V$ on $\cx F(a)_n$; hence all claims of the Theorem 
are trivially satisfied. 
This we assume that $t\ge 1$, and that our assertions are true for 
all elements of $P_n$ of dimension strictly less than $t$. 

By the definitions and Proposition~\ref{P:weak-partial-splitting}
we have for every $b\in P_n$ with $b\le a$ that 
$V_b(C_n^b\oplus M_n^b)=0$, and $V_b(N^b_{n+1})=M^b_{n+1}$, 
and $\ol{\phi(b)}(M^b_{n+1})=N^b_{n+1}$, as well as 
\begin{align*} 
\Phi_{V_b} &= 
\begin{cases} 
0     &\text{on } N^b_{n+1}\oplus M^b_n; \\ 
\id   &\text{on } C^b_n; 
\end{cases}
\\ 
V_b\ol{\phi(b)} &= 
\begin{cases} 
0     &\text{on } N^b_{n+1}\oplus C^b_n; \\
\id   &\text{on } M^b_n;  
\end{cases} 
\\
\ol{\phi(b)}V_b &= 
\begin{cases} 
0     &\text{on } C^b_n\oplus M^b_n; \\ 
\id   &\text{on } N^b_{n+1}.  
\end{cases}
\end{align*}  
Therefore 
$\phi(M_n^a)\subseteq N_{n}^a\oplus \bigoplus_{b<a} F_{n-1}^b$, and  
\[
[\id - V\phi](M_n^a)\subseteq\bigoplus_{b<a}M_n^b.
\] 
Hence  
$\Phi_V(M_n^a)\subseteq\bigoplus_{b<a}M_n^b$ and (1)  
follows from our induction hypothesis. 

Similarly, since 
$V(\cx F(a)_{n-1})= \bigoplus_{b\le a}M_n^b$, we have 
$V\phi(C^a_n)\subseteq \bigoplus_{b<a}M^b_n$. Also, $\phi V(C^a_n) = 0$ 
for trivial reasons, and (2) follows. In particular, 
$\Phi_V(C^a_n)\subseteq C^a_n\oplus \bigoplus_{b<a}M^b_n$ and the 
$C^a_n$-component of $\Phi_V(x)$ is $x$ for each $x\in C^a_n$. Therefore 
the same is true for the $C^a_n$-component of $\Phi_V^{k_a}(x)$. It is now 
immediate that the restriction 
$\Phi_V^{k_a}\: C^a_n \lra \Phi_V^{k_a}(C^a_n)$ is an isomorphism, and that  
the sum $\sum_{c\le a}\Phi^{k_c}(C^c_n)$ is a direct sum.

Next we note that (3) is an immediate consequence of 
(1) and (2), and proceed with the proof of (4). 
Since $\Phi_{V_a}(N_{n+1}^a)=0$ we get 
$\Phi_V(N_{n+1}^a)\subseteq \bigoplus_{b<a}F_n^b$, and thus 
(4) follows by our induction hypothesis using (6).  
Since $\Phi_{V_a}(F_n^a)=C^a_n$, we see that 
$\Phi_V(F_n^a)\subseteq C_n^a\oplus\bigoplus_{b<a}F_n^b$, 
hence (5) and (6) follow from (3) and the induction hypothesis. 
Now 
(7) follows from (5) and the induction hypothesis. Finally, 
(8) is clear from (1), (3), (4), and the induction 
hypothesis.  
\end{proof}

\begin{corollary}\label{C:splittings-are-minimal} 
Suppose $a$ has finite dimension in both $P_n$ and $P_{n-1}$.  
Suppose that $V_a$ is a splitting of $\ol{\cx F(a)}$, and 
that $V_c$ is a partial splitting of $\ol{\cx F(c)}$ for each 
$c\in P_n\cup P_{n-1}$ with $c\le a$. 

Then \ 
$
\phi\bigl(\Phi^{d_{P_n}(a)}_V(C^a_n)\bigr)\subseteq 
\bigoplus_{b<a}\Phi^{d_{P_{n-1}}(b)}_V(C^b_{n-1}). 
$   
\end{corollary}

\begin{proof} 
Since $\ol{\phi(a)}(C^a_n)=0$ and 
$\Phi_V(C^a_n)\subseteq C^a_n\oplus\bigoplus_{b<a} M^b_n$ 
we have $\phi\Phi_V(C^a_n)\subseteq \sum_{b<a}\cx F(b)_{n-1}$. 
As $\phi$ commutes with $\Phi_V$  
the desired conclusion is now immediate from 
Theorem~\ref{T:asymptotic}(8). 
\end{proof}

\section{Toric rings}\label{S:toric-rings}

In this section 
we address the problem of constructing minimal free resolutions 
of toric rings. These rings arise in geometry as the affine coordinate rings 
of toric varieties, and in combinatorics and other fields as 
semigroup rings of affine pointed semigroups (recall 
that a pointed affine semigroup is one that is isomorphic 
to a finitely generated subsemigroup of some $\mathbb N^m$). 
Toric rings are the subject of active current research, 
see e.g.~\cite{P, MS, F} and the references there, and 
canonical but in general non-minimal resolutions have been 
constructed in~\cite{BS, TV1}. For more details on the properties 
of toric rings and affine pointed semigroups we refer the 
reader to the excellent expositions in \cite{MS} and \cite{P}. 
We will now describe how to construct their minimal resolutions 
in a canonical and intrinsic fashion.    

Let $Q$ be an affine pointed semigroup. Then $Q$ has a unique 
minimal generating set $A$. Let $X=\{x_a \mid a\in A\}$ be a set 
of variables, and let $\Bbbk$ be a field. The map $x_a\mapsto a$ 
induces a surjective homomorphism of $\Bbbk$-algebras 
$\Bbbk[X]\lra \Bbbk[Q]$, and a natural $Q$-grading on the 
polynomial ring $\Bbbk[X]$ over $\Bbbk$ in the set of variables $X$
such that $\deg_Qx_a = a$. 
This makes the toric ring $\Bbbk[Q]$ into a $Q$-graded 
$\Bbbk[X]$-module. It is well known that $\Bbbk[Q]$ has a minimal 
free $Q$-graded resolution over $\Bbbk[X]$. The \emph{Betti degrees} 
of $Q$ over $\Bbbk$ are the elements of $Q$ that appear as $Q$-degrees 
of basis elements of the free modules in such a  minimal resolution. 
They can be computed combinatorially directly, without having to 
compute a minimal resolution first, see e.g. \cite[Theorem~9.2]{MS}. 

The semigroup $Q$ has a natural partial order given by $q\le r$ if 
and only if $q+s=r$ for some $s\in Q$. Equivalently, $q\le r$ if and 
only if there is a monomial $m$ in the variables $X$ such that 
$q+\deg_Qm = r$. This induces a partial order on the set 
$B=B(Q,\Bbbk)$ of Betti degrees of $Q$ over $\Bbbk$.
Let $\mathcal Q$ be the category with objects the elements of $Q$, with  
morphisms from $q$ to $r$ given by those monomials $m$ in the variables 
$X$ such that $q +\deg_Qm = r$. Composition of morphisms is given by 
multiplication of monomials.  The full subcategory of $\mathcal Q$ with 
objects the Betti degrees $B$ of $Q$ over $\Bbbk$ is called the 
\emph{Betti category} of $Q$ over $\Bbbk$ and denoted by 
$\mathcal B=\mathcal B(Q,\Bbbk)$.  

In \cite[Theorem~9.2]{TV1} a (non-minimal in general) canonical based finite 
free $Q$-graded resolution $\cx F$ of $\Bbbk[Q]$ 
over $\Bbbk[X]$ was constructed as follows. 
In homological degree $n$ the module $F_n$ is the free 
$\Bbbk[X]$-module with homogeneous basis given by all sequences 
$q_0\lra q_1 \lra \dots \lra q_n$ of composable nonidentity 
morphisms in the Betti category $\mathcal B(Q,\Bbbk)$, with
the $Q$-degree of the sequence being $q_n$. The differential 
$\phi_n\: F_n\lra F_{n-1}$ is defined as the sum 
$\phi_n = \sum_{i=0}^n(-1)^i\partial_i$, where the map $\partial_i$ 
sends the sequence 
\begin{equation}\label{E:morphism-sequence}
q_0\xrightarrow{m_1} q_1 
\xrightarrow{m_2} \dots \xrightarrow{m_{n-1}} 
q_{n-1}\xrightarrow{m_n} q_n
\end{equation}
to 
\[
\begin{cases}
q_0\rightarrow\dots\rightarrow q_{i-1}\xrightarrow{m_im_{i+1}}
q_{i+1}\rightarrow\dots\rightarrow q_n    
&\text{if } 0<i<n; 
\\
q_1\rightarrow\dots\rightarrow q_n 
&\text{if } i = 0; 
\\
m_n ( q_0\rightarrow\dots\rightarrow q_{n-1} ) 
&\text{if } i=n.  
\end{cases} 
\]
Using the partial order on the set $B$ of Betti degrees 
we have the natural $B$-stratification of $\cx F$ where  
$F_n^q$ is the free direct summand of $F_n$ with basis the 
set of all composable sequences  \eqref{E:morphism-sequence} 
of nonidentity morphisms in $\mathcal B$ such that $q_n=q$.  

Let us describe the corresponding stratum 
$\ol{\cx F(q)}=\cx F(q)/\sum_{p<q}\cx F(p)$. Let $\Delta_{\le q}$
(respectively $\Delta_{<q}$)  
be the semi-simplicial subset, see e.g. \cite[Sections~8.1 and~8.2]{W},   
of the nerve of $\mathcal B$, consisting of 
all sequences \eqref{E:morphism-sequence}
of composable nonidentity morphisms with $q_n\le q$ 
(respectively $q_n<q$).  It is immediate from the definitions 
that 
\[
\ol{\cx F(q)} = 
\Bbbk[X]\otimes_{\mathbb F} C_{\bullet}(\Delta_{\le q},\Delta_{<q},\mathbb F)
\]  
where 
$ 
C_{\bullet}(\Delta_{\le q},\Delta_{<q},\mathbb F) = 
 C_{\bullet}(\Delta_{\le q},\mathbb F)/ 
C_{\bullet}(\Delta_{<q},\mathbb F) 
$
is  the simplicial chain complex of the pair 
$(\Delta_{\le q},\Delta_{<q} )$ with coefficients in the prime 
field $\mathbb F$ of the field $\Bbbk$.

We are now all set to use the machinery we have developed in 
the previous sections. As a first step we have

\begin{theorem}\label{T:toric-minimal} 
Let $\mathbb F$ be the prime field of 
the field \/ $\Bbbk$. For each $q\in B$ 
let $W_q$ be a splitting of the chain complex 
$C_\bullet(\Delta_{\le q},\Delta_{<q},\mathbb F)$, and let 
$W=\sum 1\otimes W_q$ be the induced vector field on $\cx F$ 
compatible with the $B$-stratification. Let $d=\dim B$. 

Then $\Phi^{d+1}_W$ is a projection of $\cx F$ onto a free direct 
summand $\cx M_W$. Furthermore, $\cx M_W$ is a minimal 
$Q$-graded free resolution 
of the toric ring $\Bbbk[Q]$ over $\Bbbk[X]$.   
\end{theorem}

\begin{proof} 
By construction the flow $\Phi_W$ on $\cx F$ preserves the 
$Q$-grading, hence so does its iterate $\Phi^{d+1}_W$. Now 
$\Phi^{d+1}_W$ is a projection onto its image $\cx M_W$ by 
Theorem~\ref{T:asymptotic}, and, by the same theorem we 
have in homological degree $n$ that 
\[
(\cx M_W)_n = 
\bigoplus_{q\in B_n}
\Phi_W^{d(q)}(\Bbbk[X]\otimes_{\mathbb F} C^q_n). 
\] 
Thus $\cx M_W$ is chain homotopic to $\cx F$ and a $Q$-graded 
projective, hence free, direct summand of $\cx F$. In 
particular $\cx M_W$ is a $Q$-graded free resolution of $\Bbbk[Q]$ 
over $\Bbbk[X]$. 
Since elements of $\Phi_W^{d(q)}(\Bbbk[X]\otimes_{\mathbb F} C^q_n)$ 
are homogeneous of degree $q$, it is immediate from 
Corollary~\ref{C:splittings-are-minimal} that $\cx M_W$ is a 
minimal resolution as desired.    
\end{proof}

Next, we indicate how to make a ``universal'' choice for the  
splittings $W_q$ in Theorem~\ref{T:toric-minimal}. In characteristic 
zero we just take the splittings induced by taking Moore-Penrose 
pseudo-inverses. In characteristic $p>0$ 
the situation is more subtle, and we 
obtain universality only after passing to a certain purely 
transcendental extension of $\Bbbk$. The  main idea is, 
since in general 
there seems to be no canonical choice of a splitting over $\Bbbk$, 
we take the finitely many splittings over $\mathbb F_p$ and average 
them with generic weights to obtain a canonical splitting over a 
canonical transcendental extension of $\mathbb F_p$.   
In spite of its apparent simplicity, this approach provides an
esssential step in overcoming a previously unobserved obstruction
to such canonical constructions, see Remarks~\ref{R:toric-remarks}. 

The following result is an immediate consequence of 
Proposition~\ref{P:weak-partial-splitting}, 
Remarks~\ref{R:affine-combination} and~\ref{R:pseudo-inverses}, 
and Theorem~\ref{T:toric-minimal}.

\begin{theorem}\label{T:toric-canonical}
Let $\mathbb F$ be the prime subfield of $\Bbbk$. 

(a) 
Let $\mathbb F=\mathbb Q$. For each $q\in B$, 
the explicit $\mathbb Q$-basis of the homological degree $n$ component 
$C_n(\Delta_{\le q},\Delta_{<q},\mathbb Q)$ of 
$C_\bullet(\Delta_{\le q},\Delta_{<q},\mathbb Q)$,  
given by the sequences \eqref{E:morphism-sequence} with $q_n=q$,  
induces a canonical inner product for which this basis is orthonormal.
Taking $\delta_q^{+}$ to be the splitting given by the 
Moore-Penrose pseudo-inverses, see Remark~\ref{R:pseudo-inverses},  
of the differentials of 
$C_\bullet(\Delta_{\le q},\Delta_{<q},\mathbb Q)$
yields a canonical explicit vector 
field $\phi^{+}=\sum 1\otimes\delta_q^{+}$ on the chain complex $\cx F$, 
hence an explicit canonical minimal resolution $\cx M_{\phi^{+}}$ and 
an explicit canonical projection $\Phi_{\phi^{+}}^{d+1}$ of $\cx F$ onto 
$\cx M_{\phi^{+}}$.

(b)
Let \/ $\mathbb F=\mathbb F_p$ where $p\ge 2$ is a prime. For each 
$q\in B$ let $Spl_Q(q)$ be the (finite) set of all splittings of 
the chain complex $C_\bullet(\Delta_{\le q},\Delta_{<q},\mathbb F_p)$. 
Let $Y_q=\{y_\sigma \mid \sigma\in Spl_Q(q)\}$ be a set of indeterminates, 
and let $Y=\sqcup_{q\in B} Y_q$ be the disjoint union. Let 
$\mathbb F'=\mathbb F_p(Y)$ (repsectively, $\Bbbk'=\Bbbk(Y)$) 
be the purely transcendental extension 
of $\mathbb F_p$ (respectively, $\Bbbk$) on the set of variables $Y$.  
Then by Remark~\ref{R:affine-combination} 
\[
W_q = 
\frac{\sum_{\sigma\in Spl_Q(q)}\ y_{\sigma}\otimes \sigma}
     {\sum_{\sigma\in Spl_Q(q)}\ y_{\sigma}}
\]
is a canonical weak partial splitting of \/ 
$\mathbb F'\otimes C_\bullet(\Delta_{\le q},\Delta_{<q},\mathbb F_p)$, 
and taking $\widehat{W}_q$ to be the induced by 
Proposition~\ref{P:weak-partial-splitting}(bd)  canonical splitting 
yields a canonial vector 
field $W=\sum_{q\in B}1\otimes\widehat{W}_q$ on 
$\cx F'=\Bbbk'\otimes_\Bbbk\cx F$ hence a canonical minimal resolution 
$\cx M_W$ of $\Bbbk'[Q]$ over $\Bbbk'[X]$ and a canonical projection 
$\Phi_W^{d+1}$ of \/ $\cx F'$ onto $\cx M_W$.    
\end{theorem}

For the remainder of this section we elaborate on how the canonical minimal 
resolutions constructed in Theorem~\ref{T:toric-canonical} respect 
the symmetries of the pointed affine semigroup $Q$. 
In order to be more specific, 
we will use the following notation, which will come in handy also 
in the next section. 

Let $\Bbbk$ be a commutative ring, let $S$ be a $\Bbbk$-algebra, and 
let $M$ be an $S$-module. Given a homomorphism $f\: R\lra S$ of 
$\Bbbk$-algebras, we write $f_*M$ for the induced $R$-module structure 
on $M$. Thus the underlying abelian groups of $M$ and $f_*M$ are the same
set, and the $R$-module structure on $f_*M$ is given for $r\in R$ and 
$m\in M$ by $r\cdot m=f(r)m$. Clearly, if $\phi\: M\lra N$ is a 
homomorphism of $S$-modules then the same map of underlying sets is 
also a homomorphism $\phi\: f_*M\lra f_*N$ of $R$-modules. 
Furthermore, when $f$ and $g$ are endomorphisms  
of $S$ then  $(fg)_*M = g_*(f_*M)$, and 
$g\: S \lra g_* S$ is a homomorphism of $S$-modules.  

Now let $G$ be a group of automorphisms of the monoid $Q$. 
From the combinatorial description of the Betti degrees of 
$Q$ over $\Bbbk$ it is immediate that the set $B$ is $G$-invariant.  
The action of $G$ on $Q$ lifts canonically to an action of $G$ as 
automorphisms on the semigroup algebra $\Bbbk[Q]$, and since 
$G$ acts as  permutations on the minimal generating set $A$ of $Q$, 
this induces a canonical action of $G$ on the set $X$, and hence on 
the polynomial ring $\Bbbk[X]$. For $\gamma\in G$ we abuse notation 
and denote by $\gamma$ also the induced automorphisms on $\Bbbk[Q]$ 
and $\Bbbk[X]$. We consider the $\Bbbk[X]$-modules $\gamma_*\Bbbk[X]$ 
and $\gamma_*\Bbbk[Q]$ as $Q$-graded via $\deg_Q x_a=\gamma^{-1}(a)$ 
and $\deg_Q q = \gamma^{-1}(q)$ for $a\in A$ and $q\in Q$. Then the 
isomorphisms $\gamma\: \Bbbk[X]\lra\gamma_*\Bbbk[X]$ and 
$\gamma\: \Bbbk[Q]\lra\gamma_*\Bbbk[Q]$ are isomorphisms of $Q$-graded 
$\Bbbk[X]$-modules. More generally, for any $Q$-graded 
$\Bbbk[X]$-module $M$ we consider the $Q$-grading on $\gamma_*M$ 
where every homogeneous element of degree $q$ in $M$ is  
homogeneous of degree $\gamma^{-1}(q)$ in $\gamma_*M$.   

Next, the action of $G$ on $B$ and on $\Bbbk[X]$ 
induces an action on the Betti category of $Q$ over $\Bbbk$, and 
under this action the sequence \eqref{E:morphism-sequence} gets sent 
by $\gamma\in G$ to 
\[
\gamma(q_0)\xrightarrow{\gamma(m_1)} \gamma(q_1)  
\xrightarrow{\gamma(m_2)} \dots \xrightarrow{\gamma(m_{n-1})} 
\gamma(q_{n-1})\xrightarrow{\gamma(m_n)}\gamma(q_n). 
\]
Therefore $\gamma$ lifts to an isomorphism 
$
\gamma\: 
C_\bullet(\Delta_{\le q},\Delta_{<q},\mathbb F) \lra 
C_\bullet(\Delta_{\le\gamma(q)},\Delta_{<\gamma(q)},\mathbb F)
$.  
In characteristic zero this sends orthonormal bases to orthonormal 
bases, hence takes $\delta_q^{+}$ to $\delta_{\gamma(q)}^{+}$.   
In characteristic $p>0$ this 
produces a bijection of sets 
of splittings $\gamma\: Spl_Q(q)\lra Spl_Q(\gamma(q))$, and   
therefore a $G$-action on the set $Y$, 
hence also a field automorphism 
$\gamma$ of $\Bbbk'$ over $\Bbbk$; which in turn induces a 
$\Bbbk$-algebra automorphism $\gamma$ of $\Bbbk'[X]$. 
Since the module $F_n$ (respectively, $\Bbbk'\otimes_\Bbbk F_n$) 
is free over $\Bbbk[X]$ (respectively, 
$\Bbbk'[X]$) with basis all sequences \eqref{E:morphism-sequence},  
the actions of $G$ on these sequences and on $\Bbbk[X]$ (respectively, 
$\Bbbk'[X]$) induce an action on $F_n$ (respectively, 
$F'_n=\Bbbk'\otimes_\Bbbk F_n$) where $\gamma\in G$  
sends a linear combination  
$\sum c_b b$ of sequences $b$ with coefficients $c_b\in\Bbbk[X]$ 
(respectively, $c_b\in\Bbbk'[X]$) to the linear combination 
$\sum \gamma(c_b)\gamma(b)$. It is immediate that this produces 
for each $\gamma\in G$ a $\Bbbk$-automorphism $\cx F(\gamma)$ 
(respectively, $\cx F'(\gamma)$) of 
the resolution $\cx F$ (respectively, $\cx F'$) 
such that $\cx F(\gamma)\: \cx F\lra\gamma_*\cx F$ 
(respectively, $\cx F'(\gamma)\: \cx F'\lra\gamma_*\cx F'$)  
is an isomorphism of complexes of free $Q$-graded $\Bbbk[X]$-modules 
(respectively, $\Bbbk'[X]$-modules). Furthermore, it is straightforward 
to verify that $\cx F(\gamma)$ (respectively, $\cx F'(\gamma)$)
sends each $\cx F(q)$ to $\cx F(\gamma(q))$ 
(respectively, $\cx F'(q)$ to $\cx F'(\gamma(q))$),  
and commutes with the vector fields $\phi^{+}$ in characteristic zero   
(respectively, $W$ in characteristic $p>0$).       
It follows that $\cx F(\gamma)$ (respectively, $\cx F'(\gamma)$)
commutes with the flow $\Phi_{\phi^{+}}$ 
(respectively, $\Phi_W$) therefore it induces a $\Bbbk$-automorphism 
$\cx M(\gamma)$ on the minimal resolution $\cx M_{\phi^+}$ (respectively, 
$\cx M_W$).

Summarizing the discussion above, we have proved the following 
result.

\begin{theorem}\label{T:toric-intrinsic} 
Let $Q$ be a pointed affine semigroup with group of automorphisms $G$, 
let $\Bbbk$ be a field, and let the set of variables $X$ be as above. 

Then there exist 
\begin{itemize}
\item
a canonical explicitly constructed finitely generated 
field extension $\Bbbk'$ of \/ $\Bbbk$, 

\item 
a canonical explicitly constructed minimal free $Q$-graded resolution 
$\cx M$ of $\Bbbk'[Q]$ over $\Bbbk'[X]$, 

\item 
a canonical explicitly constructed 
action of $G$ as automorphisms of \/ $\Bbbk'$ over $\Bbbk$, and 

\item 
a canonical explicitly constructed homomorphism 
$\cx M\: G\lra Aut_\Bbbk(\cx M)$ of $G$ to the group of automorphisms 
of the chain complex of \/ $\Bbbk$-vector spaces $\cx M$;       
\end{itemize}
such that for each $\gamma\in G$ the map 
$\cx M(\gamma)\: \cx M \lra \gamma_*\cx M$ is an isomorphism of 
complexes of free $Q$-graded $\Bbbk'[X]$-modules that  lifts 
the isomorphism of $Q$-graded $\Bbbk'[X]$-modules  
$\gamma\: \Bbbk'[Q]\lra \gamma_*\Bbbk'[Q]$.  In characteristic zero 
one can take $\Bbbk'=\Bbbk$, and in characteristic $p>0$ one can 
take $\Bbbk'$ to be a purely transcendental extension of $\Bbbk$. 
\end{theorem}

\begin{remarks}\label{R:toric-remarks} 
(a) 
The theorem above shows that our canonical construction of 
the minimal resolution $\cx M$ uses only \emph{intrinsic} properties of $Q$, 
in the sense that it is canonically invariant under the group of symmetries 
of $Q$.  

(b) 
It is natural to ask whether in positive characteristic it is 
always possible 
to have an intrinsic construction with $\Bbbk'=\Bbbk$. At this point 
we do not have an example to the contrary. However, a similar question 
arises in the closely related case of monomial ideals, 
where we show in Theorem~\ref{T:counterexample} that in positive 
characteristic  one must have, in general, $\Bbbk'$ to be at least a 
transcendental extension of $\Bbbk$. This leads us to conjecture 
that in the characteristic $p>0$ toric case, in general, 
for an intrinsic construction of a minimal resolution one needs 
also $\Bbbk'$ to be at 
least a transcendental extension of $\Bbbk$. We will observe later 
in Remark~\ref{R:monomial-remarks}(g) that 
it is in fact possible to avoid extending the base field in all 
but finitely many positive characteristics.

(c) 
We produced our minimal resolution by starting with the canonical 
resolution $\cx F$ from \cite{TV1} that uses the normalized 
bar resolution of the Betti category of $Q$, and then constructing 
$\cx M$ as an explicit direct summand of $\cx F$. It should be clear 
by now to the reader that one could start with any other canonical 
resolution of $\Bbbk[Q]$ and proceed in an analogous fashion. For 
example, one could start with the resolution from \cite{TV1} that 
comes from the normalized bar resolution of the lub-category of $Q$, 
or, one could start with the hull resolution of $\Bbbk[Q]$, see 
\cite{BS}. 

(d)
It should be noted that, in characteristic $p>0$, the 
field extension $\Bbbk'$ 
that our method needs for the intrinsic construction will depend 
on the starting canonical resolution $\cx F$. Furthermore, 
in general the number of splittings of a nontrivial 
chain complex over $\mathbb F_p$ is quite large, which causes 
the transcendence degree of $\Bbbk'$ over $\Bbbk$ to be quite high. 
It is therefore an 
interesting problem, given $Q$, to produce a canonical starting 
resolution of $\Bbbk[Q]$ that would lead to a ``smallest'' possible 
extension $\Bbbk'$ in characteristic $p>0$.     

(e) 
It is possible, by taking into account orbits of the action of 
$G$ on the sets of splittings $Spl_Q(q)$, and some other adjustments, 
to produce versions of 
Theorem~\ref{T:toric-canonical} that in positive characteristic 
yield a field extension $\Bbbk'$ over $\Bbbk$ of a possibly smaller 
but in general  
still nonzero transcendence degree . We discuss some of these 
``other adjustments'' in the example below, and incorporate 
their precise formulation into the statements of our results 
on monomial ideals in the next section. 
\end{remarks}

\begin{example}\label{Ex:toric-minimal}  
We will use the following homologically very simple example to illustrate 
our construction.  
Let $Q$ be the numeric semigroup generated in $\mathbb N$ by $2$ and $3$. 
Note that the group of automorphisms of $Q$ is trivial. 
The corresponding toric ring $\Bbbk[Q]$ has as set of Betti degrees 
$B=\{0, 6\}$, and the corresponding polynomial ring $R=\Bbbk[x_2, x_3]$
has exactly two monomials of degree $6$, namely $x_2^3$ and $x_3^2$.  Thus 
in the Betti category there are  
only two sequences \eqref{E:morphism-sequence} of length $1$, 
namely 
\[
0\xrightarrow{x_2^3} 6, \qquad\text{ and }\qquad 
0\xrightarrow{x_3^2} 6.  
\] 
Of course, we also have two trivial sequences of length $0$, namely 
$0$ and $6$. 
The resulting resolution $\cx F$ has the form  
\[
0\lla R^2 
\xleftarrow{
\left[
\begin{smallmatrix}
-x_2^3 & -x_3^2 \\ 
1     &  1
\end{smallmatrix}
\right]
}
R^2 \lla 0.   
\]
The simplicial chain complexes of the pairs of semi-simplicial sets 
$(\Delta_{\le 6},\Delta_{<6})$ and $(\Delta_{\le 0},\Delta_{<0})$ have the 
form (over the prime field $\mathbb F$) 
\[
0\lla \mathbb F 
\xleftarrow{
\left[
\begin{smallmatrix}
1     &  1
\end{smallmatrix}
\right]
}
\mathbb F^2 \lla 0,    
\qquad\text{ and }\qquad 
0\lla \mathbb F \lla 0,    
\] 
respectively. 

(a) 
In characteristic zero, using formula \eqref{E:pseudo-inverse} gives 
\[
0\lra \mathbb F 
\xrightarrow{
\left[
\begin{smallmatrix}
1/2 \\ 
1/2
\end{smallmatrix}
\right]
}
\mathbb F^2 \lra 0,    
\]
for the Moore-Penrose splitting of 
$C_\bullet(\Delta_{\le 6},\Delta_{<6}, \mathbb F)$. The resulting vector 
field $W$ on $\cx F$ becomes 
\[
0\lra R^2 
\xrightarrow{
\left[
\begin{smallmatrix}
 0 & 1/2 \\ 
 0 & 1/2
\end{smallmatrix}
\right]
}
R^2 \lra 0,  
\]
hence the flow $\Phi_W$ is the morphism of chain complexes 
\[
\begin{CD} 
0 @<{\hphantom{\left[
\begin{smallmatrix}
1 &  (1/2)(x_2^3+x_3^2) \\
0 &  0 
\end{smallmatrix}
\right]}}<<  
R^2 
@<<<  
R^2  
@<{\hphantom{
\left[
\begin{smallmatrix}
1 &  (1/2)(x_2^3+x_3^2) \\
0 &  0 
\end{smallmatrix}
\right]
}}
<<  0  \\
@.     
@V{ 
\left[
\begin{smallmatrix}
1 &  (1/2)(x_2^3+x_3^2) \\
0 &  0 
\end{smallmatrix}
\right]}
VV     
@VV{  
\left[
\begin{smallmatrix}
 1/2  &  -1/2 \\
-1/2  &   1/2 
\end{smallmatrix}
\right]}
V     
@. \\
0 
@<{\hphantom{ 
\left[
\begin{smallmatrix}
1 &  (1/2)(x_2^3+x_3^2) \\
0 &  0 
\end{smallmatrix}
\right]
}}<<  
R^2 
@<<<  
R^2  
@<{\hphantom{
\left[
\begin{smallmatrix}
1 &  (1/2)(x_2^3+x_3^2) \\
0 &  0 
\end{smallmatrix}
\right]
}}<< 0.   
\end{CD}
\]
In this case $\Phi_W$ is already a projection and we see that  
our canonical minimal resolution $\cx M=\im\Phi_W$ is the
subcomplex of $\cx F$ 
generated by the sequence $0$ in homological degree $0$, and 
by the element 
$(0\xrightarrow{x_2^3} 6) - (0\xrightarrow{x_3^2}6)$ 
in homological degree $1$. Thus $\cx M$ has the form
\[
0 \lla R \xleftarrow{x_3^2 - x_2^3} R \lla 0
\]
with respect to this basis.

(b) 
In characteristic $p=2$ the complex 
$C_\bullet(\Delta_{\le 6}, \Delta_{<6}, \mathbb F)$ has exactly $2$ 
splittings, given by  
\[
0\lra \mathbb F 
\xrightarrow[\alpha]{
\left[
\begin{smallmatrix}
1     \\  
0
\end{smallmatrix}
\right]
}
\mathbb F^2 \lra 0,
\qquad\text{ and }\qquad     
0\lra \mathbb F 
\xrightarrow[\beta]{
\left[
\begin{smallmatrix}
0     \\  
1
\end{smallmatrix}
\right]
}
\mathbb F^2 \lra 0.     
\]
Thus we get $Y_0=\{y_0\}$ and $Y_6=\{y_\alpha, y_\beta\}$. 
Instead of taking the purely transcendental extension over $\Bbbk$ 
on the variables $y_0,y_\alpha, y_\beta$, we can ``adjust'' by taking 
$\Bbbk'$ to be the fraction field of 
$\Bbbk[y_0,y_\alpha,y_\beta]/(y_\alpha+y_\beta-1, y_0-1)$, where 
$\Bbbk[y_0, y_\alpha, y_\beta]$ 
is the polynomial ring on the variables $y_0,y_\alpha, y_\beta$,  
and we denote by $y_0,y_\alpha$ and $y_\beta$ also their images 
in $\Bbbk'$. In particular, 
$\Bbbk'$ is a purely transcendental extension of $\Bbbk$ of degree $1$. 
Let $R'=\Bbbk'[X]$ and let $\cx F'=\cx F\otimes_R R'$. 
Now since $y_\alpha + y_\beta = 1$, we get by 
Remark~\ref{R:affine-combination} that 
$\alpha\otimes y_\alpha + \beta\otimes y_\beta$ is a weak partial 
splitting on 
$\ol{\cx F'(6)} = C_\bullet(\Delta_{\le 6},\Delta_{<6},\mathbb F)\otimes R'$ 
that satisfies \eqref{E:splitting}, hence we obtain from 
Proposition~\ref{P:weak-partial-splitting}(cd) that it is in fact  
already a splitting.  Thus we obtain on $\cx F'$ the vector field $W$  
\[
0\lra R'^2 
\xrightarrow{
\left[
\begin{smallmatrix}
 0 & y_\alpha \\ 
 0 & y_\beta
\end{smallmatrix}
\right]
} 
R'^2 \lra 0,  
\]
hence the flow $\Phi_W$ is the morphism of chain complexes 
\[
\begin{CD} 
0 @<{\hphantom{\left[
\begin{smallmatrix}
1 &  y_\alpha x_2^3+ y_\beta x_3^2 \\
0 &  0 
\end{smallmatrix}
\right]}}<<  
{R'}^2 
@<<<  
{R'}^2  
@<{\hphantom{
\left[
\begin{smallmatrix}
1 &  y_\alpha x_2^3+ y_\beta x_3^2) \\
0 &  0 
\end{smallmatrix}
\right]
}}
<<  0  \\
@.     
@V{ 
\left[
\begin{smallmatrix}
1 &  y_\alpha x_2^3+ y_\beta x_3^2 \\
0 &  0 
\end{smallmatrix}
\right]}
VV     
@VV{  
\left[
\begin{smallmatrix}
 y_\beta  &  -y_\alpha \\
-y_\beta  &   y_\alpha 
\end{smallmatrix}
\right]}
V     
@. \\
0 
@<{\hphantom{ 
\left[
\begin{smallmatrix}
1 &  (1/2)(x_2^3+x_3^2) \\
0 &  0 
\end{smallmatrix}
\right]
}}<<  
{R'}^2 
@<<<  
{R'}^2  
@<{\hphantom{
\left[
\begin{smallmatrix}
1 &  (1/2)(x_2^3+x_3^2) \\
0 &  0 
\end{smallmatrix}
\right]
}}<< 0   
\end{CD}
\]
Just as in part (a), in this case $\Phi_W$ is again already a 
projection, and we obtain the intrinsic minimal resolution $\cx M$ as the 
subcomplex of $\cx F'$ generated over $R'$ again by the trivial sequence 
$0$ in homological degree $0$, and by 
$(0\xrightarrow{x_2^3}6) - (0\xrightarrow{x_3^2}6)$ 
in homological degree $1$. Thus $\cx M$ has the form
\[
0 \lla R' \xleftarrow{x_3^2 - x_2^3} R' \lla 0  
\]
with respect to this basis. 
\end{example}

\section{Monomial resolutions}\label{S:monomial-resolutions}

Let $X$ be a finite set, and let 
$
M(X) = 
\{
f\mid f\:X\lra \mathbb N \text{ is a function of sets}\}
$
be the free abelian monoid with basis the set $X$. 
Note that $M(X)$ is partially ordered by $f\le g \iff f(x)\le g(x)$ 
for every $x\in X$. 
A \emph{monomial ideal} is an ideal $I$ in the monoid 
$M(X)$. Let $G_I$ be the group of all automorphisms of 
the monoid $M(X)$ that preserve the ideal $I$. 
The properties of $I$ invariant under the action 
of $G_I$ are precisely the \emph{intrinsic} properties 
of $I$ as an ideal of the monoid $M(X)$.  

Given a field $\Bbbk$ we identify the monoid 
ring $\Bbbk[M(X)]$ with the polynomial ring $\Bbbk[X]$ 
on the set of variables $X$ via the assignment 
$f \longmapsto \prod_{x\in X}x^{f(x)}$. Under this 
identification the monomial ideal $I\subseteq M(X)$ 
generates an ideal (also denoted by $I$ and called a 
monomial ideal) in the ring $R=\Bbbk[X]$.    
The polynomial ring $R$ has a natural $\mathbb Z$-grading, 
and the monomial ideal $I$ is a homogeneous ideal
of $R$. 

Since the work of Taylor~\cite{T} it has been a main  
open problem in commutative algebra to give an explicit 
canonical description of the minimal $\mathbb Z$-graded 
free resolution of $I$ over $R$ by using only intrinsic 
properties of the ideal $I$. 
As any automorphism in $G_I$ 
is induced by a permutation on the set $X$, 
we can identify $G_I$ with  
the set of those permutations on $X$ that 
leave the set of monomials in $I$ invariant. Any such 
permutation $\sigma$ induces an automorphism of the 
$\mathbb Z$-graded $\Bbbk$-algebra $R$, denoted also 
by $\sigma$, that leaves the ideal $I$ of $R$ invariant.  
Therefore, if a canonical construction of a minimal 
free resolution $\cx M$ of $I$ uses only intrinsic 
properties of $I$ it must be invariant under the 
action of $G_I$, or, more specifically, there has to be a 
canonical group homomorphism $\cx M\: G_I\lra Aut_\Bbbk(\cx M)$  
from $G_I$ to the group of automorphisms of $\cx M$ (as a complex of 
$\Bbbk$-vector spaces) such that for each $\sigma$ 
the morphism $\cx M(\sigma)\:\cx M \lra \sigma_*\cx M$ is 
a morphism of chain complexes of $R$-modules that 
lifts the isomorphism of $R$-modules 
$\sigma\: I\lra \sigma_*I$. 

With this in mind, and considering 
the results we have already obtained in the previous section on 
the closely related case of toric rings, 
we investigate the following 
(substantially weaker than normally considered)    
version of the problem of constructing explicitly the minimal 
free resolution of a monomial ideal.   
(Recall that for a $\mathbb Z$-graded module 
$E=\bigoplus_{n\in\mathbb Z} E_n$ over a $\mathbb Z$-graded ring $S$ 
the \emph{twisted} module $E(k)$ is the $\mathbb Z$-graded $S$-module 
with homogeneous components $E(k)_n=E_{k+n}$.)

\begin{problem}\mlabel{P:mfr}
Let $I$ be a monomial ideal in $M(X)$ and let 
$R=\Bbbk[X]$ be the polynomial ring with the 
standard $\mathbb Z$-grading,
where $\Bbbk$ is a field and $X$ is a 
finite set of variables. 
Let $\m$ be the homogeneous maximal ideal of $R$ generated
by the variables. Construct in a
canonical explicit way 
a field extension $\Bbbk'$ of \/ $\Bbbk$, an action of $G_I$ as 
field automorphisms of \/ $\Bbbk'$ over $\Bbbk$, 
an acyclic chain complex 
\[
\cx M \quad = \quad  
0 \lla M_0 \lla \dots \lla 
       M_{n-1} \xleftarrow{\mu_n} M_n \lla\dots
\]
of finite free $R'=\Bbbk'[X]$-modules with 
$\mu_n(M_n)\subseteq \m M_{n-1}$ for each $n\ge 1$, 
{\bf and}  a group homomorphism 
\[
\cx M\: G_I \lra Aut_\Bbbk(\cx M)
\]
from $G_I$ to the group of \/ $\Bbbk$-automorphisms of the 
chain complex of \/ $\Bbbk$-vector spaces $\cx M$ 
such that: 
\begin{enumerate}
\item 
there exist $R'$-module isomorphisms 
$\epsilon\:\HH_0\cx M \lra IR'$ and 
\[
\theta_n\: 
M_n\lra \bigoplus_{d\in\mathbb Z} R'(-d)^{\beta_{n,d}}
\] 
for all $n\ge 0$, such that for $n\ge 1$ each map 
\[
\theta_{n-1}\mu_n\theta_n^{-1}\:
\bigoplus_{d\in\mathbb Z} R'(-d)^{\beta_{n,d}} \lra 
\bigoplus_{d\in\mathbb Z} R'(-d)^{\beta_{n-1,d}}
\]
is a morphism of $\mathbb Z$-graded $R'$-modules, and,  
with respect to the corresponding induced 
$\mathbb Z$-grading  on $\HH_0\cx M$, so is the 
map $\epsilon$.

\item 
there exists an $R'$-module isomorphism 
$\delta\: \HH_0\cx M\lra IR'$ such that for each 
$\sigma\in G_I$ the automorphism 
\[
\cx M(\sigma)\: \cx M \lra \sigma_*\cx M
\]
is a morphism of chain complexes of $R'$-modules, 
and the induced isomorphism 
$\bar\sigma\: \HH_0\cx M \lra \sigma_*\HH_0\cx M$ 
satisfies $\sigma\delta = \delta\bar\sigma$. 
\end{enumerate} 
\end{problem}

\begin{remark}\label{R:weaker-formulation}
Notice that we do {\bf not} ask for 
\begin{itemize} 
\item
a $M(X)$-graded structure on $\cx M$, i.e. we do not 
require a monomial-, fine-, or multi-grading on $\cx M$;  

\item 
an explicit or canonical construction of the 
isomorphisms $\epsilon$, and $\theta_n$, i.e. 
of the $\mathbb Z$-graded structure of each $M_n$;   
in particular, we do {\bf not} ask for a canonical or explicit 
construction of a homogeneous basis of each $M_n$; 

\item 
an explicit or canonical construction of the 
isomorphism $\delta$, i.e. of an identification 
of the zeroth homology of $\cx M$ with the ideal $IR'$;   

\item 
the automorphisms $\cx M(\sigma)$ to preserve any  
$\mathbb Z$-graded structure of $\cx M$. 
\end{itemize}  
In fact we will show later in Theorem~\ref{T:counterexample}
that in characteristic $p>0$ our problem has, in general, 
no solution with 
$\Bbbk=\mathbb F_p$ and $\Bbbk'$ an algebraic extension, 
even in this very weak formulation. 
\end{remark}

In characteristic zero, a solution 
goes back at least to the work of Yuzvinsky~\cite{Y}. 
It is given as a special case of the main result
\cite[Theorem~4.1]{Y}, by taking the splitting 
$W_a$ of each 
complex $\ol{\cx T(a)}$ in Example~\ref{Ex:Taylor-4} to be 
$\delta_a^{+}$ as in Example~\ref{Ex:Taylor-0}, 
the splitting given by taking Moore-Penrose pseudo-inverses. 
Yuzvinsky's construction produces the minimal resolution $\cx M$ 
as a subcomplex of the Taylor resolution, and as indicated 
in Examples~\ref{Ex:Taylor-3} and~\ref{Ex:Taylor-4}, 
Yuzvinsky's proof can be leveraged 
to also produce an explicit projection of the Taylor resolution 
onto $\cx M$. 

In their recent paper \cite{EMO}, Eagon, Miller, and Ordog leverage   
their ingenious combinatorial description (the ``Hedge formula'') 
of Moore-Penrose pseudo-inverses of differentials 
in cellular chain complexes, 
and Eagon's construction \cite{E} of a Wall complex, applied to a 
certain double complex, to 
give another solution in characteristic zero that has a strong 
combinatorial component, and works also in all but finitely many 
positive characteristics. 

In the rest of this section we use 
our dynamical systems theory and 
present a solution to Problem~\ref{P:mfr} that works in all 
characteristics. The construction keeps the base field unchanged 
in all except certain finitely many positive characteristics 
(which depend on the ideal), where it produces a transcendental 
base field extension before obtaining the resolution.  
Our motivation comes from the influential work \cite{GPW}, 
where Gasharov, Peeva, and Welker show 
that the isomorphism class of the lcm-lattice of a monomial 
ideal determines the structure of the minimal resolution, in 
the sense that an isomophism between the lcm-lattices of two 
ideals determines a canonical ``relabeling'' procedure that 
transforms a minimal resolution of one ideal into a minimal
resolution of the other. It is therefore natural to ask whether
it is possible to construct an explicit canonical 
minimal resolution of a monomial 
ideal directly from its lcm-lattice in an intrinsic fashion.  
We will now describe how to accomplish this. 

Let $U$ be the set of minimal generators of our monomial ideal 
$I$ in $M(X)$. Recall from Example~\ref{Ex:Taylor-1} that the 
lcm-lattice $L$ of $I$ is the join semi-lattice join-generated in $M(X)$ 
by the set $U$. Equivalently, representing the elements of $M(X)$ 
as monomials in the polynomial ring $\Bbbk[X]$ via the 
identification $f\mapsto m_f=x^f = \prod_{y\in X}y^{f(y)}$, 
the lcm-lattice $L$ is represented by all monomials that are  
least common multiples of subsets of 
the set of monomials $\{m_u\mid u\in U\}$. In this representation 
the partial order on $L$ is precisely the one given by 
divisibility of monomials. 

Next, set $L'=L\setminus\{1\}$, obtained by removing from $L$ its 
minimal element $\hat 0 = 1$, and 
consider the order simplicial complex $\Delta=\Delta(L')$ 
of the poset $L'$. The $n$-faces of $\Delta$ are given 
by all chains  
\begin{equation}\label{E:n-faces}
a_0< \dots < a_n
\end{equation} 
in $L'$, in particular  
they have a natural orientation induced by the ordering on their vertices. 
Thus we obtain canonical bases for the spaces of oriented $n$-chains  
$C_n(\Delta, \mathbb F)$ and 
$
C_n(a) := C_n(\Delta_{\le a}, \Delta_{<a}, \mathbb F)=
C_n(\Delta_{\le a},\mathbb F)/C_n(\Delta_{<a},\mathbb F)
$, 
where $\mathbb F$ is the prime field of $\Bbbk$, and the 
abstract simplicial complex 
$\Delta_{\le a}$ (respectively, $\Delta_{<a}$) consists of all faces 
\eqref{E:n-faces} of $\Delta$ with $a_n\le a$ (resectively, $a_n<a$).   
In particular, in charateristic zero we obtain a natural inner product 
on these spaces that makes these bases orthonormal. 
Choosing a subset $X_i$ of this canonical basis of 
$C_i(a)=C_i(\Delta_{\le a}, \Delta_{<a},\mathbb F)$ such that it maps under 
the differential $d_i^a$ of the simplicial chain complex 
$C_\bullet(\Delta_{\le a},\Delta_{<a},\mathbb F)$ 
of the pair $(\Delta_{\le a},\Delta_{<a})$ 
bijectively onto a basis of $\im d_i^a$ induces a unique map 
$\tilde d_i^a\: \im d_i^a\lra C_i(a)$ 
satisfying $\tilde d_i^a d_i^a\big\vert_{X_i}=\id_{X_i}$. Furthermore, for each 
canonical basis element $b$ of $C_i(a)$ 
the unique expression $d_i^a(b)=\sum_{c\in X_i}r_cd_i^a(c)$ produces a  
canonical $i$-cycle $z(b)=z_{X_i}(b)=b-\sum_{c\in X_i}r_cc$ in $\Ker d_i^a$. 
It is a routine exercise in linear algebra to show that the set 
$K_i=\{z(b)\mid b\notin X_i\}$ forms a basis of $\Ker d_i^a$, therefore 
any subset $Z_i$ of $K_i$ that maps bijectively to a basis of 
$H_i(a)=\HH_i(\Delta_{\le a},\Delta_{<a},\mathbb F)$ under the canonical 
projection $\pi_i\: \Ker d_i^a \lra H_i(a)$ induces a canonical map 
$\tilde\pi_i\: H_i(a)\lra C_i(a)$ 
satisfying $\tilde\pi_i \pi_i\big\vert_{Z_i}=\id_{Z_i}$. Therefore choosing 
such a pair $(X_i,Z_i)$ for each $i\ge 0$ induces canonically a splitting of 
$C_\bullet(\Delta_{\le a},\Delta_{<a},\mathbb F)$, and we call such a splitting 
a \emph{matroidal splitting}.         
Let $MSpl(a)$ be the set of all matroidal splittings of 
$C_\bullet(\Delta_{\le a}, \Delta_{<a}, \mathbb F)$. 
Clearly 
\begin{equation}\label{E:matroidal-bound}
m(a):= |MSpl(a)| \le 
\prod_{i\ge 0}\binom{\dim C_i(a)}{\rank d_i^a}
            \binom{\dim C_i(a)-\rank d_i^a}{\dim H_i(a)}. 
\end{equation}
We call $a\in L'$ \emph{critical} over $\Bbbk$ if $char(\Bbbk)$ divides 
$m(a):= |MSpl(a)|$. We write $L_{crit}'$ for the set of critical elements 
of $L'$ (over $\Bbbk$). In particular, if $char(\Bbbk)=0$ then 
$L'_{crit}=\emptyset$, i.e. there are no critical elements.    
For each $a\in L'_{crit}$  we let  
\[
Y_a=\{y_\sigma \mid \sigma\in MSpl(a)\}
\]
be a set of variables.  
Let $Y=\sqcup_{a\in L'_{crit}}Y_a$ be the disjoint union. 
Note that $Y$ is finite. 
We set 
\begin{equation}\label{E:field-extension}
\Bbbk' = 
\Bbbk(Y)^{\widetilde{}} 
\end{equation}
where $\Bbbk(Y)^{\widetilde{}}$ is the
field of fractions of the domain 
$\Bbbk[Y]^{\widetilde{}}=\Bbbk[Y]/(l_a\mid a\in L'_{crit})$
with $l_a=1-\sum_{\sigma\in MSpl(a)}y_\sigma$ for each $a\in L'_{crit}$.
Thus $\Bbbk'$ is a purely transcendental extension of $\Bbbk$, of
transcendence degree $|Y|-|L'_{crit}|$, 
and $\Bbbk'=\Bbbk$ when $char(\Bbbk)=0$. 

It is well-known that the simplicial complex $\Delta$ supports 
a $M(X)$-graded free resolution $\cx F=(F_n,\phi_n)$ of $IR'$ 
over $R'=\Bbbk'[X]$, called the \emph{lcm-resolution} of $IR'$, 
see \cite[Proposition 3.3.27]{OW}, described as follows. 
In homological degree $n\ge 0$ the 
module $F_n$ is the free $M(X)$-graded 
$R'$-module with basis the set of all 
oriented $n$-faces of $\Delta$, each oriented face 
\eqref{E:n-faces} homogeneous of degree $a_n\in L'$. 
For $n\ge 1$ the differential $\phi_n\: F_n\lra F_{n-1}$ is given 
by sending an oriented face $T$ as in \eqref{E:n-faces} to  
\[
\sum_{i=0}^{n-1} (-1)^i T_i 
\quad + \quad  
(-1)^n (x^{a_n}/x^{a_{n-1}}) T_n ,    
\]
where for each $0\le i\le n$ the face 
$T_i$ is the oriented face obtained by omitting $a_i$ from $T$. 
We have a canonical $L'$-stratification on $\cx F$ where for each 
$a\in L'$ the module 
$F_n^a$ is the free $R'$-submodule of $F_n$ generated by the 
oriented faces \eqref{E:n-faces} with $a_n=a$.  It is immediate from 
this description that for the stratum at $a$ we have 
\[
\ol{\cx F(a)}=
C_\bullet(\Delta_{\le a},\Delta_{<a},\mathbb F)\otimes R'. 
\] 

Now all that remains is to specify a splitting 
$W_a$ on each $\ol{\cx F(a)}$ 
and apply Theorem~\ref{T:asymptotic}. In characteristic zero 
we can take $W_a=\delta_a^+\otimes 1$, where  $\delta_a^+$ is the 
splitting as in  Remark~\ref{R:pseudo-inverses} 
given by taking the Moore-Penrose inverses of the differentials 
of $C_\bullet(\Delta_{\le a},\Delta_{<a},\mathbb F)$ with respect 
to the natural inner products. 
Alternatively, in \emph{any} charateristic,  
we can take 
$W_a= \widehat{V_a}$, where 
\[
V_a=
\begin{cases} 
\sum_{\sigma\in MSpl(a)}\sigma\otimes y_\sigma  &\text{ if } a\in L'_{crit}; \\
\frac{1}{m(a)}
\sum_{\sigma\in MSpl_(a)} \sigma\otimes 1      &\text{ otherwise; }
\end{cases} 
\] 
is, by Remark~\ref{R:affine-combination}, a weak partial splitting on 
$C_\bullet(\Delta_{\le a}, \Delta_{<a},\mathbb F)\otimes R'$ 
satisfying \eqref{E:splitting}, and $\widehat{V_a}$ is the 
induced splitting from Proposition~\ref{P:weak-partial-splitting}.    

With this we are all set to use our dynamical systems machinery. 
By Theorem~\ref{T:asymptotic} the flow $\Phi_W$ of the vector field 
$W=\sum_{a\in L'}W_a$ stabilizes after $1 + \dim L'$ iterations, hence 
$\Phi_W^{1+\dim L'}$ is a projection of $\cx F$ onto a projective 
hence free resolution $\cx M$ of $IR'$ over $R'$. Since each $W_a$ 
preserves $M(X)$-degrees, so do $\Phi_W$ and its iterates, therefore 
$\cx M$ is $M(X)$-graded, and therefore minimal by 
Corollary~\ref{C:splittings-are-minimal}.   

Finally, since under the induced action of $G_I$ 
an element $\gamma\in G_I$ sends 
the canonical basis of $C_i(a)$ to the canonical basis of 
$C_i\bigl(\gamma(a)\bigr)$, we have $G_I$ sending matroidal 
splittings to matroidal splittings. Therefore  
constructing the action of $G_I$ on $\Bbbk'$ and the 
homomorphism of groups $\cx M\: G_I \lra Aut_\Bbbk(\cx M)$ is 
\emph{mutatis mutandis} that of the corresponding construction 
for the group $G$ in the proof of Theorem~\ref{T:toric-intrinsic}. 

We have thus completed a proof of

\begin{theorem}\label{T:monomial-intrinsic} 
Let $I$ be a monomial ideal in the free abelian monoid $M(X)$ on a  
finite set $X$, with group of automorphisms $G_I$, and 
let $\Bbbk$ be a field. 

Then there exist 
\begin{itemize}
\item
a canonical explicitly constructed finitely generated
field extension $\Bbbk'$ of \/ $\Bbbk$, 

\item 
a canonical explicitly constructed $M(X)$-homogeneous 
projection of the lcm-resolution 
$\cx F$ of $I\Bbbk'[X]$ over $\Bbbk'[X]$ onto a 
direct summand $\cx  M$ which is a  
minimal free $M(X)$-graded resolution 
of $I\Bbbk'[X]$ over $\Bbbk'[X]$, 

\item 
a canonical explicitly constructed 
action of $G_I$ as automorphisms of \/ $\Bbbk'$ over $\Bbbk$, and 

\item 
a canonical explicitly constructed homomorphism 
$\cx M\: G_I\lra Aut_\Bbbk(\cx M)$ from $G_I$ to the group of automorphisms 
of the chain complex of \/ $\Bbbk$-vector spaces $\cx M$;       
\end{itemize}
such that for each $\gamma\in G_I$ the map 
$\cx M(\gamma)\: \cx M \lra \gamma_*\cx M$ is an isomorphism of 
complexes of free $M(X)$-graded $\Bbbk'[X]$-modules that  lifts 
the isomorphism of $M(X)$-graded $\Bbbk'[X]$-modules  
$\gamma\: I\Bbbk'[X]\lra \gamma_*(I\Bbbk'[X])$.  In characteristic zero 
one can take $\Bbbk'=\Bbbk$, and in characteristic $p>0$ one can 
take $\Bbbk'$ to be a purely transcendental extension of $\Bbbk$ of 
degree $\sum_{a\in L'_{crit}} (m(a)-1)$. 
\end{theorem}

\begin{corollary}\label{C:infinite-primes}
Given any monomial ideal $I$ in $M(X)$,  Problem~\ref{P:mfr} 
has a solution with $\Bbbk'=\Bbbk=\mathbb F_p$ for every sufficiently 
big prime $p$. 
\end{corollary}

\begin{proof} 
By \eqref{E:matroidal-bound} we get for each $a\in L'$ 
\[
m(a)\le \prod_{i\ge 0}2^{\dim C_i(a)}2^{\dim C_i(a)}
\]
hence for $p\ge \max\{\prod_{i\ge 0}4^{\dim C_i(a)}\mid  a\in L'\}$ we 
get the desired conclusion. 
\end{proof}

\begin{remarks}\label{R:monomial-remarks} 
(a)
By now it should be clear to the reader that instead of using the
lcm-resolution or the Taylor resolution as a starting point, one
could apply our technique starting from any other canonical construction
of a based free resolution of a monomial ideal $I$. For example, one could
opt for the (in general smaller than the lcm-resolution) resolution supported
on the order complex of the Betti poset of $I$, see e.g. \cite{TV}. 

(b)
Our construction provides an intrinsic
canonical set of $M(X)$-homogeneous generators for the minimal resolution
$\cx M$, but not a basis. In general it is not possible to produce a
canonical and intrinsic (not even up to associates) basis of $\cx M$, not
even in characteristic zero. This was already pointed out in 
\cite[Example 3.8]{EMO},
and that observation continues to hold also under the much weaker requirements
of Probem~\ref{P:mfr}. One can provide a rigorous proof by using an appropriately
modified version of the proof of Theorem~\ref{T:counterexample}. We leave
these details to the interested reader. 

(c)
Our construction has additional functorial
properties not mentioned in Theorem~\ref{T:monomial-intrinsic}. 
Given a monomial ideal $J\subset M(X)$ and an element $m\in M(X)$, 
write $J_{\le m}$ for the monomial ideal generated by the those 
minimal generators of $J$ that are $\le m$.
Denote by $\mathcal I(X)$ the category with 
objects the monomial ideals in $M(X)$,
and with morphisms from $I$ to $J$ those automorphisms 
$\gamma$ of $M(X)$ such that $\gamma(I)=J_{\le m}$ for some 
(depending on $\gamma$) monomial $m$.
We leave to the reader the routine verification that our construction
produces assignments $I \mapsto \cx M=\cx M(I)$ and
$\gamma\mapsto \cx M(\gamma)$ that make it into a functor
$\cx M\:\mathcal I(X) \lra \comp$, where $\comp$ is the category of chain
complexes of finite free modules as described in \cite[Section 1]{Tc}.  

(d)
Just as in the toric rings case, the transcendence degree of the base
field extension that our method requires in positive characteristic depends
on the starting non-minimal canonical resolution,
see Example~\ref{Ex:mfr-example}(a). For instance, for ideals of low
projective dimension but with a sufficiently big number of minimal generators,
starting with the lcm-resolution will in general produce a lower transcendence
degree than starting with the Taylor resolution. It is an interesting problem
whether there exists a canonical non-minimal starting construction that would
produce for every ideal a transcendence degree that is lowest among the degrees 
produced by all other canonical non-minimal starting constructions.

(e)
Given a monomial ideal $I$ and a prime $p\ge 2$,
a natural question to ask is what is
the smallest transcendence degree required of an extension $\Bbbk'$ of
$\Bbbk=\mathbb F_p$ so that Problem~\ref{P:mfr} has a solution for $I$.
Another related question is whether there is in some sense a 
``minimal'' such transcendental extension, perhaps unique up to an isomorphism.
At this point it seems that a functorial construction 
such as ours would not be able to produce that kind of ``minimal''
extension for every monomial ideal.  
It is very likely that one would
need to sacrifice at least some of the functoriality described in 
part (c) above, 
in order to capture the specifics of a given ideal 
that would lead to such a  ``minimal'' extension. 

(f) 
Given a monomial ideal $I$, another interesting problem is to describe 
the set of all primes $p\ge 2$ for which Problem~\ref{P:mfr} has no 
solution for $I$ with $\Bbbk'=\Bbbk=\mathbb F_p$.  

(g) 
By taking averages as in the proof of Theorem~\ref{T:monomial-intrinsic} 
of matroidal splittings only (instead of all splittings), 
one obtains a proof of Theorem~\ref{T:toric-intrinsic} that keeps 
the base field unchanged in all but finitely many positive characteristics. 

(h)
A matroidal splitting is the same as a splitting induced by 
what is called a \emph{community} in \cite{EMO}. 
\end{remarks}

\section{An example}\label{S:an-example}

To illustrate our construction, in this section we compute the
following very useful 

\begin{example}\label{Ex:mfr-example}  
Let $\Bbbk$ be a field and let $\mathbb F$ 
be the prime subfield of $\Bbbk$. 
Let $X=\{v_0,v_1,v_2,v_3,e_{12},e_{23},e_{31}\}$ and let $I$ be the 
monomial ideal generated 
by the set of monomials 
\[
m_0 = v_1v_2v_3e_{12}e_{23}e_{31} ,\quad  
m_1 = v_0^2v_2v_3e_{23} ,\quad  
m_2 = v_0^2v_1v_3e_{31} ,\quad 
m_3 = v_0^2v_1v_2e_{12}. 
\]
Note this is precisely the ideal $I(3)$ from
Definition~\ref{D:the-counterexample}. 
The lcm-lattice of $I$ is the set of monomials 
$\{1, m_0, m_1, m_2, m_3, m_{12}, m_{31}, m_{23}, m\}$, 
where 
\begin{align*}
m     &= v_0^2v_1v_2v_3e_{12}e_{23}e_{31}, \\  
m_{12} &= v_0^2v_1v_2v_3e_{23}e_{31},      \\ 
m_{23} &= v_0^2v_1v_2v_3e_{31}e_{12},      \\ 
m_{31} &= v_0^2v_1v_2v_3e_{12}e_{23}.         
\end{align*}
In order to have uniform notation, we also set $e_{ij}=e_{ji}$ 
and $m_{ij}=m_{ji}$. The simplicial complex $\Delta=\Delta(L')$ 
has $8$ vertices, $13$ edges, and six $2$-faces.  In the canonical 
bases given by the oriented faces, the complex 
$C_\bullet(\Delta_{\le m},\Delta_{<m},\mathbb F)$ has the form 
\[
0 \lla 
\mathbb F 
\xleftarrow{
[
\begin{smallmatrix}
1 & 1 & 1 & 1 & 1 & 1 & 1 
\end{smallmatrix}
]
} 
\mathbb F^{7} 
\xleftarrow{
\left[
\begin{smallmatrix}
\hm 0 &\hm 0 &\hm 0 &\hm 0 &\hm 0 &\hm 0  \\ 
   -1 &   -1 &\hm 0 &\hm 0 &\hm 0 &\hm 0  \\ 
\hm 0 &\hm 0 &   -1 &   -1 &\hm 0 &\hm 0  \\ 
\hm 0 &\hm 0 &\hm 0 &\hm 0 &   -1 &   -1  \\ 
\hm 1 &\hm 0 &\hm 1 &\hm 0 &\hm 0 &\hm 0  \\ 
\hm 0 &\hm 1 &\hm 0 &\hm 0 &\hm 1 &\hm 0  \\
\hm 0 &\hm 0 &\hm 0 &\hm 1 &\hm 0 &\hm 1
\end{smallmatrix}
\right]
}
\mathbb F^{6} 
\lla 0, 
\]
each of the complexes  
$C_\bullet(\Delta_{\le m_{ij}},\Delta_{<m_{ij}},\mathbb F)$ for 
$1\le i<j\le 3$  
has the form 
\[
0\lla \mathbb F
\xleftarrow{
[
\begin{smallmatrix}
1 & 1 
\end{smallmatrix}
]
} 
\mathbb F^2 
\lla 0,  
\] 
and each of the complexes 
$C_\bullet(\Delta_{\le m_i},\Delta_{<m_i},\mathbb F)$ 
for $i=0,\dots,3$  
has the form 
\[
0\lla \mathbb F\lla 0. 
\]

(a) 
A standard computation shows that in every characteristic 
we have $72$ matroidal splittings of 
the complex $C_\bullet(\Delta_{\le m},\Delta_{<m},\mathbb F)$,  
we have $2$ matroidal splittings for each of the three complexes 
$C_\bullet(\Delta_{\le m_{ij}}, \Delta_{<m_{ij}},\mathbb F)$, and 
only one matroidal splitting for each of the four complexes 
$C_\bullet(\Delta_{\le m_i}, \Delta_{<m_i}, \mathbb F)$. Thus using 
our methods with the lcm-resolution as a starting point we 
produce an intrinsic minimal resolution over $\mathbb F_p$ whenever 
$p\ge 5$. When $p=2$ we produce an intrinsic minimal resolution 
over a purely 
transcendental extension of $\mathbb F_2$ of transcendence 
degree $|Y| - |L'_{crit}| = (72-1) + 3(2-1)=74$, 
and when $p=3$ we get a minimal resolution over a purely transcendental 
extension of $\mathbb F_3$ of transcendence degree $72-1=71$.  
We leave it to the reader ro check that the corresponding computation with 
the Taylor resolution as starting point produces an intrinsic 
construction in both $p=2$ and $p=3$ cases 
after taking a transcendental extension with 
transcendence degree ``only'' $17$. On the other hand, 
when $p\ne 3$ there 
is an intrinsic construction for the minimal resolution of $I$   
without extending the field at all, see Remark~\ref{R:intrinsic-resolution}(a).
However, by Theorem~\ref{T:counterexample},  
in characteristic $p=3$ it is not possible to  avoid taking
some kind of transcendental 
extension. Remark~\ref{R:intrinsic-resolution}(b) gives such an
intrinsic construction using an extension of $\mathbb F_3$
of transcendence degree $2$. At this point 
we do not know if it is possible to have an intrinsic construction
where the transcendence degree of the required field extension
of $\mathbb F_3$ is exactly $1$. 

(b) 
In characteristic zero taking Moore-Penrose inverses and using 
the explicit formula \eqref{E:pseudo-inverse} yields the 
following splittings: 
\[
0\lra \mathbb F 
\xrightarrow{
\left[ 
\begin{smallmatrix} 
1/7 \\
1/7 \\
1/7 \\
1/7 \\
1/7 \\
1/7 \\ 
1/7
\end{smallmatrix}
\right] 
}
\mathbb F^7 
\xrightarrow{
1/12 
\left[
\begin{smallmatrix} 
      0& {-5}&\hm 3& {-1}&\hm 5& {-3}&\hm 1\\
      0& {-5}& {-1}&\hm 3& {-3}&\hm 5&\hm 1\\
      0&\hm 3& {-5}& {-1}&\hm 5&\hm 1& {-3}\\
      0& {-1}& {-5}&\hm 3& {-3}&\hm 1&\hm 5\\
      0&\hm 3& {-1}& {-5}&\hm 1&\hm 5& {-3}\\
      0& {-1}&\hm 3& {-5}&\hm 1& {-3}&\hm 5
\end{smallmatrix}
\right]
}
\mathbb F^6 \lra 0   
\]
for $C_\bullet(\Delta_{\le m}, \Delta_{< m}, \mathbb F)$, see 
Example~\ref{Ex:pseudo-inverses},  and 
\[
0 \lra \mathbb F 
\xrightarrow{
\left[ 
\begin{smallmatrix} 
1/2 \\
1/2 
\end{smallmatrix}
\right] 
}
\mathbb F^2 \lra 0 
\]
for each of the complexes 
$C_\bullet(\Delta_{\le m_{ij}}, \Delta_{< m_{ij}}, \mathbb F)$ with 
$1\le i<j\le 3$.  The lcm-resolution $\cx F$ has the form  
{\small
\begin{align*}
0 \leftarrow 
R^8 
\xleftarrow[\phi_1]{
-\left[
\begin{smallmatrix}
\,\, v_0^2&    \hm 0&   \hm 0&         \hm 0
     &    \hm 0&   \hm 0&         \hm 0 
     &    \hm 0&   \hm 0&         \hm 0
     &    \hm 0&   \hm 0&         \hm 0  \\ 
\hm 0& v_1e_{31}& v_1e_{12}& v_1e_{12}e_{31}
     &    \hm 0&   \hm 0&         \hm 0 
     &    \hm 0&   \hm 0&         \hm 0
     &    \hm 0&   \hm 0&         \hm 0  \\ 
\hm 0&    \hm 0&   \hm 0&         \hm 0
     & v_2e_{23}& v_2e_{12}& v_2e_{12}e_{23}
     &    \hm 0&   \hm 0&         \hm 0
     &    \hm 0&   \hm 0&         \hm 0  \\ 
\hm 0&    \hm 0&   \hm 0&         \hm 0
     &    \hm 0&   \hm 0&         \hm 0 
     & v_3e_{23}& v_3e_{31}& v_3e_{31}e_{23}
     &    \hm 0&   \hm 0&         \hm 0  \\ 
\hm 0&       -1&   \hm 0&         \hm 0
     &       -1&   \hm 0&         \hm 0 
     &    \hm 0&   \hm 0&         \hm 0
     &    e_{12}&   \hm 0&         \hm 0  \\ 
\hm 0&    \hm 0&      -1&         \hm 0
     &    \hm 0&   \hm 0&         \hm 0 
     &       -1&   \hm 0&         \hm 0
     &    \hm 0&   e_{31}&         \hm 0  \\ 
\hm 0&    \hm 0&   \hm 0&         \hm 0
     &    \hm 0&      -1&         \hm 0 
     &    \hm 0&      -1&         \hm 0
     &    \hm 0&   \hm 0&         e_{23}  \\ 
   -1&    \hm 0&   \hm 0&            -1
     &    \hm 0&   \hm 0&            -1  
     &    \hm 0&   \hm 0&            -1
     &       -1&      -1&            -1     
\end{smallmatrix}
\right]
} 
&R^{13} 
\\[+5pt] 
\xleftarrow[\phi_2]{
\left[
\begin{smallmatrix} 
\hm 0        &    \hm 0&    \hm 0& 
\hm 0        &    \hm 0&    \hm 0          \\ 
e_{12}        &    \hm 0&    \hm 0& 
\hm 0        &    \hm 0&    \hm 0          \\ 
\hm 0        &    e_{31}&    \hm 0& 
\hm 0        &    \hm 0&    \hm 0          \\ 
-1           &       -1&    \hm 0& 
\hm 0        &    \hm 0&    \hm 0          \\ 
\hm 0        &    \hm 0&    e_{12}& 
\hm 0        &    \hm 0&    \hm 0          \\ 
\hm 0        &    \hm 0&    \hm 0& 
e_{23}        &    \hm 0&    \hm 0          \\ 
\hm 0        &    \hm 0&       -1& 
-1           &    \hm 0&    \hm 0          \\ 
\hm 0        &    \hm 0&    \hm 0& 
\hm 0        &    e_{31}&    \hm 0          \\ 
\hm 0        &    \hm 0&    \hm 0& 
\hm 0        &    \hm 0&    e_{23}          \\ 
\hm 0        &    \hm 0&    \hm 0& 
\hm 0        &       -1&       -1          \\ 
\hm 1        &    \hm 0&    \hm 1& 
\hm 0        &    \hm 0&    \hm 0          \\ 
\hm 0        &    \hm 1&    \hm 0& 
\hm 0        &    \hm 1&    \hm 0          \\ 
\hm 0        &    \hm 0&    \hm 0& 
\hm 1        &    \hm 0&    \hm 1             
\end{smallmatrix} 
\right]
}
&R^6 \leftarrow 0 
\end{align*}
}
and the  
resulting vector field $W=\phi^+$ on $\cx F$ is given by 
\begin{align*}
0 \lra 
R^8 
&\xrightarrow[\phi_1^+]{
\left[
\begin{smallmatrix} 
0            &    \hm 0&    \hm 0&   \hm 0&
\hm 0        &    \hm 0&    \hm 0&     1/7       \\ 
0            &    \hm 0&    \hm 0&   \hm 0&
1/2          &    \hm 0&    \hm 0&   \hm 0       \\ 
0            &    \hm 0&    \hm 0&   \hm 0&
\hm 0        &      1/2&    \hm 0&   \hm 0       \\ 
0            &    \hm 0&    \hm 0&   \hm 0&
\hm 0        &    \hm 0&    \hm 0&     1/7       \\ 
0            &    \hm 0&    \hm 0&   \hm 0& 
1/2          &    \hm 0&    \hm 0&   \hm 0       \\ 
0            &    \hm 0&    \hm 0&   \hm 0&
\hm 0        &    \hm 0&      1/2&   \hm 0       \\ 
0            &    \hm 0&    \hm 0&   \hm 0& 
\hm 0        &    \hm 0&    \hm 0&     1/7       \\ 
0            &    \hm 0&    \hm 0&   \hm 0& 
\hm 0        &      1/2&    \hm 0&   \hm 0       \\ 
0            &    \hm 0&    \hm 0&   \hm 0&  
\hm 0        &    \hm 0&      1/2&   \hm 0       \\ 
0            &    \hm 0&    \hm 0&   \hm 0&  
\hm 0        &    \hm 0&    \hm 0&     1/7       \\ 
0            &    \hm 0&    \hm 0&   \hm 0&  
\hm 0        &    \hm 0&    \hm 0&     1/7       \\ 
0            &    \hm 0&    \hm 0&   \hm 0&  
\hm 0        &    \hm 0&    \hm 0&     1/7       \\ 
0            &    \hm 0&    \hm 0&   \hm 0&  
\hm 0        &    \hm 0&    \hm 0&     1/7          
\end{smallmatrix} 
\right]
}
\\[+5pt]
R^{13} &\xrightarrow[\phi_2^+]{
1/12 
\left[
\begin{smallmatrix} 
      0&\hm 0&\hm 0& {-5}&\hm 0&\hm 0&\hm 3
       &\hm 0&\hm 0& {-1}&\hm 5& {-3}&\hm 1\\
      0&\hm 0&\hm 0& {-5}&\hm 0&\hm 0& {-1}
       &\hm 0&\hm 0&\hm 3& {-3}&\hm 5&\hm 1\\
      0&\hm 0&\hm 0&\hm 3&\hm 0&\hm 0& {-5}
       &\hm 0&\hm 0& {-1}&\hm 5&\hm 1& {-3}\\
      0&\hm 0&\hm 0& {-1}&\hm 0&\hm 0& {-5}
       &\hm 0&\hm 0&\hm 3& {-3}&\hm 1&\hm 5\\
      0&\hm 0&\hm 0&\hm 3&\hm 0&\hm 0& {-1}
       &\hm 0&\hm 0& {-5}&\hm 1&\hm 5& {-3}\\
      0&\hm 0&\hm 0& {-1}&\hm 0&\hm 0&\hm 3
       &\hm 0&\hm 0& {-5}&\hm 1& {-3}&\hm 5
\end{smallmatrix}
\right] 
}
R^6 \longrightarrow 0. 
\end{align*}
Therefore the flow 
$\Phi=\Phi_W=\Phi_{\phi^{+}}=\id - \phi\phi^+ - \phi^+\phi$ 
is given by
\[
\begin{CD}
0 @<<< R^8 @< \phi_1 <<  R^{13} @< \phi_2 <<  R^6  @<<< 0 \\
@.    @V \Phi_0 VV    @V \Phi_1 VV       @V \Phi_2 VV @. \\ 
0 @<<< R^8 @< \phi_1 <<  R^{13} @< \phi_2 <<   R^6 @<<<  0,  
\end{CD}
\]
where 
\[
\Phi_0=
\left[
\begin{smallmatrix}
\hm 1\hm&\hm 0\hm&\hm 0\hm&\hm 0\hm&\hm 0\hm&\hm 0\hm&\hm 0\hm&\frac{v_0^2}{7}  
\\[+5pt]
\hm 0\hm&\hm 1\hm&\hm 0\hm&\hm 0\hm&\frac{v_1e_{13}}{2}&
\frac{v_1e_{12}}{2}&\hm 0\hm&\frac{v_1e_{12}e_{31}}{7}
\\[+5pt]
\hm 0\hm&\hm 0\hm&\hm 1\hm&\hm 0\hm&\frac{v_2e_{23}}{2}&\hm 0\hm&
\frac{v_2e_{12}}{2}&\frac{v_2e_{12}e_{23}}{7}
\\[+5pt]
\hm 0\hm&\hm 0\hm&\hm 0\hm&\hm 1\hm&\hm 0\hm&\frac{v_3e_{13}}{2}&
\frac{v_3e_{13}}{2}&\frac{v_3e_{13}e_{23}}{7}
\\[+5pt]
\hm 0\hm&\hm 0\hm&\hm 0\hm&\hm 0\hm&\hm 0\hm&\hm 0\hm&\hm 0\hm&\frac{e_{12}}{7}
\\[+5pt]
\hm 0\hm&\hm 0\hm&\hm 0\hm&\hm 0\hm&\hm 0\hm&\hm 0\hm&\hm 0\hm&\frac{e_{13}}{7}
\\[+5pt]
\hm 0\hm&\hm 0\hm&\hm 0\hm&\hm 0\hm&\hm 0\hm&\hm 0\hm&\hm 0\hm&\frac{e_{23}}{7}
\\[+5pt]
\hm 0\hm&\hm 0\hm&\hm 0\hm&\hm 0\hm&\hm 0\hm&\hm 0\hm&\hm 0\hm&\hm 0\hm 
\end{smallmatrix}
\right], 
\]

\[
\Phi_1 = 
{
\left[
\begin{smallmatrix}
\frac{6}{7}&0&0&{\frac{-1}{7}}&0&0&{\frac{-1}{7}}&0&0&
{\frac{-1}{7}}&{\frac{-1}{7}}&{\frac{-1}{7}}&{\frac{-1}{7}}
\\[+5pt]
0&\frac{1}{2}&0&\frac{5e_{12}}{12}&{\frac{-1}{2}}&0&
{\frac{-e_{12}}{4}}&0&0&\frac{e_{12}}{12}&\frac{e_{12}}{12}&
\frac{e_{12}}{4}&{\frac{-e_{12}}{12}}
\\[+5pt]
0&0&\frac{1}{2}&\frac{5e_{31}}{12}&0&0&\frac{e_{31}}{12}&
{\frac{-1}{2}}&0&{\frac{-e_{31}}{4}}&\frac{e_{31}}{4}&
\frac{e_{31}}{12}&{\frac{-e_{31}}{12}}
\\[+5pt]
{\frac{-1}{7}}&0&0&\frac{1}{42}&0&0&\frac{1}{42}&0&0&\frac{1}{42}&
\frac{1}{42}&\frac{1}{42}&\frac{1}{42}
\\[+5pt]
0&{\frac{-1}{2}}&0&{\frac{-e_{12}}{4}}&\frac{1}{2}&0&
\frac{5e_{12}}{12}&0&0&\frac{e_{12}}{12}&\frac{e_{12}}{12}&
{\frac{-e_{12}}{12}}&\frac{e_{12}}{4} 
\\[+5pt]
0&0&0&\frac{e_{23}}{12}&0&\frac{1}{2}&\frac{5e_{23}}{12}&0&
{\frac{-1}{2}}&{\frac{-e_{23}}{4}}&\frac{e_{23}}{4}&
{\frac{-e_{23}}{12}}&\frac{e_{23}}{12}
\\[+5pt]
{\frac{-1}{7}}&0&0&\frac{1}{42}&0&0&\frac{1}{42}&0&0&\frac{1}{42}&
\frac{1}{42}&\frac{1}{42}&\frac{1}{42}
\\[+5pt]
0&0&{\frac{-1}{2}}&{\frac{-e_{31}}{4}}&0&0&\frac{e_{31}}{12}&
\frac{1}{2}&0&\frac{5e_{31}}{12}&{\frac{-e_{31}}{12}}&
\frac{e_{31}}{12}&\frac{e_{31}}{4}
\\[+5pt]
0&0&0&\frac{e_{23}}{12}&0&{\frac{-1}{2}}&{\frac{-e_{23}}{4}}&0&
\frac{1}{2}&\frac{5}{12}\,{x}_{1}&{-\frac{1}{12}\,{x}_{1}}&
\frac{e_{23}}{4}&\frac{e_{23}}{12}
\\[+5pt]
{\frac{-1}{7}}&0&0&\frac{1}{42}&0&0&\frac{1}{42}&0&0&\frac{1}{42}&
\frac{1}{42}&\frac{1}{42}&\frac{1}{42}
\\[+5pt]
{\frac{-1}{7}}&0&0&\frac{1}{42}&0&0&\frac{1}{42}&0&0&\frac{1}{42}&
\frac{1}{42}&\frac{1}{42}&\frac{1}{42}
\\[+5pt]
{\frac{-1}{7}}&0&0&\frac{1}{42}&0&0&\frac{1}{42}&0&0&\frac{1}{42}&
\frac{1}{42}&\frac{1}{42}&\frac{1}{42}
\\[+5pt]
{\frac{-1}{7}}&0&0&\frac{1}{42}&0&0&\frac{1}{42}&0&0&\frac{1}{42}&
\frac{1}{42}&\frac{1}{42}&\frac{1}{42}
\end{smallmatrix}
\right], 
}
\]
and 
\[
\Phi_2 = (1/6)  
\left[
\begin{smallmatrix}
      \hm 1& {-1}& {-1}&\hm 1&\hm 1& {-1}\\
       {-1}&\hm 1&\hm 1& {-1}& {-1}&\hm 1\\
       {-1}&\hm 1&\hm 1& {-1}& {-1}&\hm 1\\
      \hm 1& {-1}& {-1}&\hm 1&\hm 1& {-1}\\
      \hm 1& {-1}& {-1}&\hm 1&\hm 1& {-1}\\
       {-1}&\hm 1&\hm 1& {-1}& {-1}&\hm 1
\end{smallmatrix}
\right]. 
\] 
The posets $L'_n$ indexing the nonzero components
of the $L'$-stratification of $\cx F$ in homological degree $n$ are 
\begin{align*}
L'_0 &=L', \\
L'_1 &=\{m_{12}, m_{31}, m_{23}, m\}, \text{ and }\\ 
L'_2 &=\{m\}.
\end{align*}  
Thus, according to Theorem~\ref{T:asymptotic} we must have that 
$\Phi_0^3$, $\Phi_1^2$, and $\Phi_2$ are projections. In fact
even $\Phi_0^2$ is already a projection. Computing these 
matrices gives   
\[
\Phi_0^2 = 
\left[
\begin{smallmatrix}
\hm 1\hm&\hm 0\hm&\hm 0\hm&\hm 0\hm&\hm 0\hm&\hm 0\hm&
\hm 0\hm&\frac{v_0^2}{7}
\\[+5pt]
\hm 0\hm&\hm 1\hm&\hm 0\hm&\hm 0\hm&\frac{v_1e_{31}}{2}&
\frac{v_1e_{12}}{2}&\hm 0\hm&\frac{2v_1e_{12}e_{31}}{7}
\\[+5pt]      
\hm 0\hm&\hm 0\hm&\hm 1\hm&\hm 0\hm&\frac{v_2e_{23}}{2}&
\hm 0\hm&\frac{v_2e_{12}}{2}&\frac{2v_2e_{12}e_{23}}{7}
\\[+5pt]
\hm 0\hm&\hm 0\hm&\hm 0\hm&\hm 1\hm&\hm 0\hm&\frac{v_3e_{23}}{2}&
\frac{v_3e_{31}}{2}&\frac{2v_3e_{23}e_{31}}{7}
\\[+8pt]
\hm 0\hm&\hm 0\hm&\hm 0\hm&\hm 0\hm&\hm 0\hm&\hm 0\hm&\hm 0\hm&\hm 0\hm 
\\[+10pt]
\hm 0\hm&\hm 0\hm&\hm 0\hm&\hm 0\hm&\hm 0\hm&\hm 0\hm&\hm 0\hm&\hm 0\hm 
\\[+10pt]
\hm 0\hm&\hm 0\hm&\hm 0\hm&\hm 0\hm&\hm 0\hm&\hm 0\hm&\hm 0\hm&\hm 0\hm 
\\[+10pt]
\hm 0\hm&\hm 0\hm&\hm 0\hm&\hm 0\hm&\hm 0\hm&\hm 0\hm&\hm 0\hm&\hm 0\hm 
\end{smallmatrix}
\right], 
\]
and 
{\small 
\[
\Phi_1^2 = 
\left[
{
\begin{smallmatrix}
\frac{6}{7}&0&0&{\frac{-1}{7}}&0&0&{\frac{-1}{7}}
&0&0&{\frac{-1}{7}}&{\frac{-1}{7}}&{\frac{-1}{7}}&{\frac{-1}{7}}   
\\[+5pt]      
{\frac{-e_{12}}{14}}&\frac{1}{2}&0&\frac{29e_{12}}{84}&
{\frac{-1}{2}}&0&{\frac{-9e_{12}}{28}}&0&0&\frac{e_{12}}{84}&
\frac{e_{12}}{84}&\frac{5e_{12}}{28}&{\frac{-13e_{12}}{84}}
\\[+5pt]      
{\frac{-e_{13}}{14}}&0&\frac{1}{2}&\frac{29e_{13}}{84}&0&0&
\frac{e_{13}}{84}&{\frac{-1}{2}}&0&{\frac{-9e_{13}}{28}}&
\frac{5e_{13}}{28}&\frac{e_{13}}{84}&{\frac{-13e_{13}}{84}}
\\[+5pt]      
{\frac{-1}{7}}&0&0&\frac{1}{42}&0&0&\frac{1}{42}&0&0&
\frac{1}{42}&\frac{1}{42}&\frac{1}{42}&\frac{1}{42}
\\[+5pt]      
{\frac{-e_{12}}{14}}&{\frac{-1}{2}}&0&{\frac{-9e_{12}}{28}}&
\frac{1}{2}&0&\frac{29e_{12}}{84}&0&0&\frac{e_{12}}{84}&
\frac{e_{12}}{84}&{\frac{-13e_{12}}{84}}&\frac{5e_{12}}{28}
\\[+5pt]      
{\frac{-e_{23}}{14}}&0&0&\frac{e_{23}}{84}&0&\frac{1}{2}&
\frac{29e_{23}}{84}&0&{\frac{-1}{2}}&{\frac{-9e_{23}}{28}}&
\frac{5e_{23}}{28}&{\frac{-13e_{23}}{84}}&\frac{e_{23}}{84}
\\[+5pt]
{\frac{-1}{7}}&0&0&\frac{1}{42}&0&0&\frac{1}{42}&0&0&\frac{1}{42}&
\frac{1}{42}&\frac{1}{42}&\frac{1}{42}
\\[+5pt]
{\frac{-e_{13}}{14}}&0&{\frac{-1}{2}}&{\frac{-9e_{13}}{28}}&0&0&
\frac{e_{13}}{84}&\frac{1}{2}&0&\frac{29e_{13}}{84}&
{\frac{-13e_{13}}{84}}&\frac{e_{13}}{84}&\frac{5e_{13}}{28}
\\[+5pt]
{\frac{-e_{23}}{14}}&0&0&\frac{e_{23}}{84}&0&{\frac{-1}{2}}&
{\frac{-9e_{23}}{28}}&0&\frac{1}{2}&\frac{29e_{23}}{84}&
{\frac{-13e_{23}}{84}}&\frac{5e_{23}}{28}&\frac{e_{23}}{84}
\\[+5pt]
{\frac{-1}{7}}&0&0&\frac{1}{42}&0&0&\frac{1}{42}&0&0&\frac{1}{42}&
\frac{1}{42}&\frac{1}{42}&\frac{1}{42}
\\[+5pt]
{\frac{-1}{7}}&0&0&\frac{1}{42}&0&0&\frac{1}{42}&0&0&\frac{1}{42}&
\frac{1}{42}&\frac{1}{42}&\frac{1}{42}
\\[+5pt]
{\frac{-1}{7}}&0&0&\frac{1}{42}&0&0&\frac{1}{42}&0&0&\frac{1}{42}&
\frac{1}{42}&\frac{1}{42}&\frac{1}{42}
\\[+5pt]
{\frac{-1}{7}}&0&0&\frac{1}{42}&0&0&\frac{1}{42}&0&0&\frac{1}{42}&
\frac{1}{42}&\frac{1}{42}&\frac{1}{42}
\end{smallmatrix}
}
\right]. 
\] 
}
\negthickspace\negthinspace
Thus the minimal resolution $\cx M$, being the image of $\Phi^3=\Phi^2$, 
is provided by our method with a canonical (but non-minimal) intrinsic
$M(X)$-homogeneous generating set. While in general one cannot hope for
much better, for this particular ideal it is straightforward to see from the
form of the matrices above that $\cx M$ 
has as $M(X)$-homogeneous bases 
\[
\bigl\{\quad h_0=\{m_0\},\ h_1=\{m_1\},\ h_2=\{m_2\},\ h_3 =\{m_3\}\quad \bigr\}
\] 
in homological degree $0$, 
\begin{align*}
\Bigl\{
\quad 
g_{12} = \frac{1}{2}\bigl(\{m_1,m_{12}\} &- \{m_2,m_{12}\}\bigr), \\[+3pt]
g_{31} = \frac{1}{2}\bigl(\{m_3,m_{31}\} &- \{m_1,m_{31}\}\bigr), \\[+3pt]
g_{23} = \frac{1}{2}\bigl(\{m_2,m_{23}\} &- \{m_3,m_{23}\}\bigr), \\[+3pt] 
g_0   = \frac{1}{14}\bigl(12\{m_0,m\}   &- e_{12}\{m_1,m_{12}\} - e_{12}\{m_2,m_{12}\} \\[+9pt] 
                       &- e_{31}\{m_3,m_{31}\} - e_{31}\{m_1,m_{31}\}    \\[+12pt] 
                       &- e_{23}\{m_2,m_{23}\} - e_{23}\{m_3,m_{23}\}    \\[+12pt]
                       &- 2\{m_1,m\} - 2\{m_2,m\} - 2\{m_3,m\}         \\[+8pt]
                       &- 2\{m_{12},m\} - 2\{m_{31},m\} - 2\{m_{23},m\}\bigr)  
\quad
\Bigr\}
\end{align*}
in homological degree $1$, and 
\begin{multline*} 
\Bigl\{\quad 
f = \frac{1}{6}\bigl(
\{m_1,m_{12},m\} - \{m_1,m_{31},m\} - \{m_2,m_{12},m\} 
+\{m_2,m_{23},m\} \\ 
+\{m_3,m_{31},m\} - \{m_3,m_{23},m\}\bigr)  
\quad\Bigr\}
\end{multline*}
in homological degree $2$. With respect to these bases $\cx M$ takes
the form
\[
0\lla R^4
\xleftarrow{
\left[  
\begin{smallmatrix}
\frac{-6}{7}v_0^2         &                \hm 0&
                    \hm 0&               \hm 0      \\ 
\frac{2}{7}v_1e_{12}e_{31}&\frac{-1}{2}v_1e_{31}&
\frac{1}{2}v_1e_{12}     &               \hm 0      \\ 
\frac{2}{7}v_2e_{12}e_{23}&\frac{1}{2}v_2e_{23} &
                    \hm 0&\frac{-1}{2}v_2e_{12}       \\ 
\frac{2}{7}v_3e_{31}e_{23}&               \hm 0&
\frac{-1}{2}v_3e_{23}      &\frac{1}{2}v_3e_{31}          
\end{smallmatrix}    
\right]
}
R^4 \xleftarrow{
\frac{1}{3}
\left[
\begin{smallmatrix}
       0      \\ 
      e_{12}  \\ 
      e_{31}  \\ 
      e_{23}     
\end{smallmatrix}    
\right]
}
R \lla 0,  
\]
and thus, up to associates, this is exactly the construction from
Remark~\ref{R:intrinsic-resolution}(a). 
Note that the automorphism group of 
our ideal is $S_3$, where $\tau\in S_3$ acts on $X$ by fixing $v_0$, 
and sending $v_i$ and $e_{ij}$ to $v_{\tau(i)}$ and $e_{\tau(i)\tau(j)}$
respectively.     
The basis of $\cx M$ above is invariant 
up to sign under the induced action of $S_3$ on the lcm-resolution $\cx F$.   
Furthermore, this choice of basis is canonical up to sign: 
in homological degree $0$ this is the only
basis-forming subset of the set of canonical 
generators; in homological degree $1$ each 
$g_{ij}$ is up to sign the only canonical generator of degree $m_{ij}$, 
and $g_0$ is the only canonical generator of degree $m$ that is fixed 
by $S_3$; and in homological degree $2$ up to sign $f$ is the only 
canonical generator. 
\end{example}

\section{Monomial resolutions in positive characteristic}
\label{S:monomial-resolutions-in-positive-characteristic}

Throughout this section $p\ge 2$ is a fixed prime. 

We show that there is a monomial ideal that does not have a   
minimal intrinsic resolution  over any algebraic extension of $\mathbb F_p$.

\begin{definition}\label{D:the-counterexample}
We set
\[
n=n(p)=
\begin{cases}
4 &\text{ if } p=2; \\   
p &\text{ otherwise}.
\end{cases}
\]
We write $I=I(p)$ for the 
monomial ideal on the set of $2n+1$ variables 
\[
X=X(p)=\{v_0,\dots,v_n, e_{1,2},e_{2,3},\dots, e_{n-1,n}, e_{n,1}\}
\]
generated by the set of monomials 
$
\{
m_0,m_1,\dots, m_n 
\},  
$ 
where \ 
$
m= 
\prod_{y\in X}y
$ 
\ and 
\begin{gather*}
m_0=\frac m {v_0},  \qquad 
m_1=\frac {v_0m} {e_{n,1}v_1e_{1,2}}, \qquad 
m_n=\frac {v_0m} {e_{n-1,n}v_ne_{n,1}}, \\ 
\text{ and } \qquad 
m_i= \frac {v_0m} {e_{i-1,i}v_ie_{i,i+1}} \qquad \text{ for } 
i=2,\dots, n-1.
\end{gather*}
Also, for convenience of notation, for each 
$e_{i,j}\in X(p)$ we set $e_{j,i}=e_{i,j}$. 
\end{definition}

The main result in this section is  

\begin{theorem}\mlabel{T:counterexample} 
Problem~\ref{P:mfr} 
does not have a solution for the ideal $I(p)$ with 
$\Bbbk=\mathbb F_p$ and $\Bbbk'$ an algebraic field 
extension of \ $\mathbb F_p$. 
\end{theorem}

While not immediately apparent from the proof presented here, 
the fundamental fact behind this theorem is that cyclic $p$-groups have 
nontrivial cohomology over $\mathbb F_p$. Also, most of the 
work in the proof is to reduce to the case of group action 
on the minimal resolution that preserves the $\mathbb Z$-grading. 
For this reason we postpone the proof till the very end of the 
section, and 
proceed with all the necessary results for this desired reduction.   
Our first task is to describe the group of symmetries of $I(p)$, 
and its minimal free resolution. 

\begin{remarks}\label{R:trivial-action}
Let $\Bbbk$ be a field, and let $R=\Bbbk[X]$ 
be the polynomial ring over $\Bbbk$ in the set of variables $X=X(p)$. 

(a) 
Inside the symmetric group $S_n$ consider the dihedral 
subgroup $D_n$ generated by 
$\sigma=(1,n)(2,n-1)\cdots(\lceil n/2\rceil, \lceil (n+1)/2\rceil)$ 
and $\rho=(1,\dots,n)$. The action of $S_n$ 
on the set $\{1,\dots, n \}$ induces an action of $D_n$ on the 
set $X$ where $\tau\in D_n$ fixes $v_0$, sends 
$v_i$ to $v_{\tau(i)}$, and sends $e_{i,j}$ to 
$e_{\tau(i),\tau(j)}$ when $i,j\ge 1$.  
It is immediate from this formula that  $D_n$ 
preserves the monomial ideal $I(p)$, and it is straightforward 
to check that there are no other automorphisms of $M(X)$ that 
preserve $I(p)$. 

(b) 
Suppose that $char(\Bbbk) = p$, and that 
we have an action of $D_n$ as ring 
automorphisms on the polynomial ring $R$ that 
extends the action of $D_n$ on the set $X$. Since the 
units of $R$ are the nonzero elements of the field 
$\Bbbk$ this yields an action of $D_n$ as field 
automorphisms on $\Bbbk$. 
If $x\in \Bbbk$ is algebraic over $\mathbb F_p$ and is 
not fixed by $D_n$, then we look at the splitting field 
$\widetilde\Bbbk=\mathbb F_p(x)$ of $x$ over $\mathbb F_p$.  
Consider the normal subgroup 
$T$ of $D_n$ consisting of the elements of $D_n$ that 
act trivially on $\widetilde\Bbbk$. It follows that $D_n/T$ is 
a subgroup of the Galois group of a finite extension 
of $\mathbb F_p$, hence is cyclic. Therefore our 
assumptions on $n$ imply that $\rho$ is in $T$, i.e. 
$\rho$ acts trivially on $x$. Thus if $\Bbbk$ is 
algebraic over $\mathbb F_p$ we have that $\rho$ acts 
trivially on $\Bbbk$. 
\end{remarks}

\begin{lemma}\mlabel{L:minimal-resolution}
Let \/ $\Bbbk$ be any field and let $R=\Bbbk[X]$. 

(a) 
The lcm-lattice $L$ of $I$ consists of the following elements: 
\[
L=\{\ 
1, \ m_0,\ \dots\ ,\ m_n,\ v_0m/e_{1,2},\ \dots\ ,\ v_0m/e_{n-1,n},\ 
                         v_0m/e_{n,1},\ v_0m \ \}. 
\]

(b)
The minimal free resolution of $R/I$ over $R$ 
has the form 
\[
0\lla R \overset{\phi_0}\lla H 
        \overset{\phi_1}\lla G \overset{\phi_2}\lla F \lla 0.
\]
Also, there are homogeneous bases $\{h_0,h_1,\dots, h_n\}$, 
$\{g_0, g_{1,2}, \dots, g_{n-1,n}, g_{n,1}\}$, and 
$\{f\}$ of $H$, $G$, and $F$, respectively, 
of degrees 
\begin{equation}\label{E:basis-degrees}
\begin{split}
|h_0| & = 2n ; \\
|h_1| = \dots = |h_n| & = 2n-1; \\ 
|g_{1,2}| = \dots = |g_{n-1,n}| = |g_{n,1}| & = 2n+1; \\ 
|g_0|= |f| & = 2n+2,   
\end{split}
\end{equation}
and such that 
the maps $\phi_j$ satisfy the following equalities: 
\begin{equation}\label{E:differentials}
\begin{split} 
\phi_0(h_i)     &= m_i \quad\text{for each }i; \\ 
\phi_1(g_0)     &= v_0^2h_0 - e_{n,1}v_1e_{1,2} h_1;  \\ 
\phi_1(g_{1,2})  &= v_2e_{2,3} h_2 - e_{n,1}v_1 h_1; \\ 
\phi_1(g_{n,1})  &= v_1e_{1,2} h_1  - e_{n-1,n}v_n h_n ; \\ 
\phi_1(g_{n-1,n}) &= v_ne_{n,1} h_n - e_{n-2,n-1}v_{n-1} h_{n-1}; \\
\phi_1(g_{i,i+1}) &= v_{i+1}e_{i+1,i+2} h_{i+1} - e_{i-1,i}v_i h_i 
                   \quad\text{for}\quad 3\le i\le n-2; \\
\phi_2(f)       &= e_{1,2}g_{1,2} + e_{2,3}g_{2,3}+ \dots + e_{n,1}g_{n,1}.    
\end{split}
\end{equation} 
\end{lemma}

\begin{proof} 
(a) 
The claim on the lcm-lattice is straightforward to verify. 

(b) 
This can be proven by taking the equalities 
\eqref{E:differentials} and \eqref{E:basis-degrees} as 
the definition of $\cx F$, 
and checking directly that this gives a resolution of $R/I$. 
However, we prefer the 
following more conceptual proof.   

Consider the set of variables $X'=(X\cup\{v_0^2\})\setminus\{v_0\}$, 
and let $R'=\Bbbk[X']$ be the polynomial ring over $\Bbbk$ on the set 
of variables $X'$. Thus $R'$ is a subring of $R$, and $R$ is a free, 
hence flat $R'$-module. 
Let $\Delta$ be the simplicial complex on the vertex set 
$\{v_0^2,v_1,\dots, v_n\}$, with oriented facets $\{v_0^2\}$,  
$e_{1,2}=\{v_1, v_2\}, \dots, e_{n-1,n}=\{v_{n-1}, v_n\}$, and 
$e_{n,1}=\{v_n, v_1\}$. Thus $\Delta$ is just the disjoint union of 
a point and an $n$-cycle.  Let $I'$ be  
the nearly-Scarf ideal of $\Delta$ in $R'$, in the sense of 
Peeva and Velasco~\cite{PV}. Note that $I$ is exactly the ideal 
generated in $R$ by $I'$. Choose $\{v_0^2\} - \{v_1\}$ as an oriented cycle
generating $\RH_0(\Delta,\Bbbk)$, and $\{e_{12}+e_{23}+\dots + e_{n,1}\}$ as
an oriented cycle generating $\RH_1(\Delta,\Bbbk)$, and  
let $\cx F'$ be 
the corresponding minimal free resolution 
of $R'/I'$ over $R'$, as described in \cite[Theorem 6.1]{PV}.  
From that description it is immediate to check that the complex 
$\cx F := R\otimes_{R'}\cx F'$ has homogeneous bases satisfying 
all desired equalities. As $\cx F$ is a resolution 
of $R/I$ over $R$ by the flatness of $R$, this completes our proof.  
\end{proof}

The statements and proofs of all remaining results in this section  
will adhere strictly to the following convention:

\begin{convention}
If a field $\Bbbk$ is specified, then $\cx F$ always
stands for the minimal free resolution
\[
0\lla H \xleftarrow{\phi_1} G \xleftarrow{\phi_2} F \lla 0 
\]
of $I=I(p)$ over $R=\Bbbk[X]$ with bases and maps as described in
Lemma~\ref{L:minimal-resolution}. For uniform notation, 
we set $g_{ji}=-g_{ij}$. Furthermore, we will always assume
we are given an action of $D_n$ as automorphisms of $R$ that extends
the action of $D_n$ on $X$. For each $\tau\in D_n$ 
we will denote by $\tau$ also the induced automorphism of 
$R$. Clearly, $\tau$ is a symmetry of $I$  
such that $\tau(m_i)=m_{\tau(i)}$ for each $i$, 
and $\tau\: I \lra \tau_*I$ is an isomorphism of $R$-modules.  
\end{convention}

\begin{remarks}\label{R:intrinsic-resolution} 
(a) 
Supppose $\Bbbk$ is a field of characteristic $\ne p$.  
In view of Lemma~\ref{L:minimal-resolution} 
it is straightforward to check that the 
complex 
\[
0\lla H \xleftarrow{\tilde\phi_1} 
                            G \xleftarrow{\phi_2} F \lla 0
\]
is a minimal resolution $\cx M$ of $I(p)$, where $\tilde\phi_1$ agrees with 
$\phi_1$ on all basis elements of $G$ except on $g_0$ where we have 
\begin{align*}
\tilde\phi_1(g_0) = 
v_0^2h_0 - (1/n)\bigl(e_{n,1}v_1e_{1,2} h_1  &+ e_{1,2}v_2e_{2,3} h_2 + \dots \\
 &+ e_{n-2,n-1}v_{n-1}e_{n-1,n} h_{n-1} + e_{n-1,n}v_ne_{n,1} h_n\bigr). 
\end{align*} 
Furthermore, the action of $\tau\in D_n$ on $R$ and $I$ lifts 
naturally to an action  
$\cx M(\tau)$ on $\cx M$ that fixes the basis elements $f, g_0$, and $h_0$,  
and for $1\le i<j\le n$ sends $h_i$ and $g_{ij}$ to $h_{\tau(i)}$ and 
$g_{\tau(i)\tau(j)}$, respectively. Thus $\cx M$ gives a solution to 
Problem~\ref{P:mfr} for $I(p)$ that has $\Bbbk'=\Bbbk$ in characteristic 
$\ne p$.     

(b) 
Suppose now that $char(\Bbbk)=p$. Let $Y=\{y_1,\dots, y_n\}$ be a 
set of variables, and let $\Bbbk'$ be the fraction field of the 
domain $\Bbbk[Y]/J$, where $\Bbbk[Y]$ is the polynomial ring over 
$\Bbbk$ in the set of variables $Y$, and $J$ is the principal ideal 
generated by the linear form $y_1+\dots +y_n -1$. We use $y_i$ to denote 
also the image of $y_i$ in $\Bbbk'$. Each $\tau\in D_n$ acts on $Y$ via 
$\tau(y_i)= y_{\tau(i)}$. Since this leaves $J$ invariant, it induces 
an action on $\Bbbk'$ given by the same formula. Thus we get an action 
on $R'=\Bbbk'[X]$ that leaves $I(p)R'$ invariant. Just as in part (a), 
for each $\tau\in D_n$ we obtain an action $\cx M'(\tau)$ on the 
$R'$-modules $H'=H\otimes_R R'$, $G'=G\otimes_R R'$ and 
$F'=F\otimes_R R'$. This gives an action $\cx M'(\tau)$ on the chain 
complex $\cx M'$ given by 
\[
0\lla H' \xleftarrow{\phi_1'} 
                            G' \xleftarrow{\phi_2\otimes 1} F' \lla 0
\]
where $\phi_1'$ agrees with $\phi_1\otimes 1$ on the basis elements 
$g_{ij}\otimes 1$, and satisfies 
\begin{align*}
\phi_1'(g_0\otimes 1) = 
v_0^2h_0\otimes 1 &- e_{n,1}v_1e_{1,2} h_1\otimes y_1 -  
                         e_{1,2}v_2e_{2,3} h_2\otimes y_2 - \dots \\
                  &- e_{n-2,n-1}v_{n-1}e_{n-1,n} h_{n-1}\otimes y_{n-1} - 
    e_{n-1,n}v_ne_{n,1} h_n\otimes y_n . 
\end{align*} 
It is now straightforward to check that $\cx M'$ gives a solution 
to Problem~\ref{P:mfr} for the ideal $I(p)$ with $\Bbbk'$ a purely 
transcendental extension of $\Bbbk$ of degree $n-1$. 
\end{remarks}

\begin{lemma}\mlabel{L:unique-action-ideal}
Let $\tau\in D_n$, and 
suppose $\Bbbk$ is any field.  
Let $\psi\: I \lra \tau_*I$ be an isomorphism 
of $R$-modules.

Then $\psi=c\tau$ for some $c\in\Bbbk$. In particular, 
if $\psi\: I \lra I$ is an isomorphism, then $\psi=c\id$ 
for some $c\in\Bbbk$. 
\end{lemma}

\begin{proof}
Note that for $i=2,\dots,n-1$ one has the relation 
$v_0^2m_0= e_{i-1,i}v_ie_{i,i+1}m_i$ in $I$. Therefore 
we get $v_0^2\psi(m_0)=e_{\tau(i-1)\tau(i)}v_{\tau(i)}e_{\tau(i)\tau(i+1)}\psi(m_i)$.
Smilarly, the relations $v_0^2m_0=e_{n,1}v_1e_{1,2}m_1$ and 
$v_0^2m_0=e_{n-1,n}v_ne_{n,1}m_n$ yield the relations 
\begin{align*}
v_0^2\psi(m_0) &=e_{\tau(n)\tau(1)}v_{\tau(1)}e_{\tau(1)\tau(2)}\psi(m_1) \\
v_0^2\psi(m_0) &=e_{\tau(n-1)\tau(n)}v_{\tau(n)}e_{\tau(n)\tau(1)}\psi(m_n).
\end{align*} 
It follows that 
$\psi(m_0)=cm_0$ for some $c\in R$. 
Therefore 
\begin{align*}
\psi(m_1)&= cv_0^2m_0/(e_{\tau(n)\tau(1)}v_{\tau(1)}e_{\tau(1)\tau(2)})
         =cm_{\tau(1)} \\ 
\psi(m_n)&= cv_0^2m_0/(e_{\tau(n-1)\tau(n)}v_{\tau(n)}e_{\tau(n)\tau(1)})
         =cm_{\tau(n)}.
\end{align*} 
Finally, for 
$2\le i\le n-1$ we have 
\[
\psi(m_i)=cv_0^2m_0/(e_{\tau(i-1)\tau(i)}v_{\tau(i)}e_{\tau(i)\tau(i+1)})
         =cm_{\tau(i)}, 
\] 
therefore $\psi=c\tau$. Since $\psi$ is an epimorphism, $c$ is a unit 
in $R$, hence $c\in\Bbbk$. 
\end{proof}

\begin{lemma}\label{L:graded-lowest-degrees} 
Let $\Bbbk$ be any field, 
let $\tau$ be any degree preserving ring automorphism of $R$,
let $M$ be any finitely generated graded $R$-module, and let 
$\gamma\: M \lra \tau_* M$ be an isomorphism of $R$-modules 
such that $\gamma^m =\id$ for some $m\ge 1$. 
For any $0\ne u\in M$ write $u^*$ for 
the nonzero homogeneous component of $u$ of lowest degree. 
Let $\{w_1,\dots, w_k\}$ be a homogeneous generating set of $M$ 
such that $|\gamma(w_i)^*| = |w_i|$ for each $i$. 

Then for each nonzero $u\in M$ and each $s\ge 1$ we have 
$|\gamma^s(u)^*| = |u^*|$.   
\end{lemma}

\begin{proof} 
Note first that for every nonzero homogeneous $u$ and every 
$s\ge 1$ we have 
$|\gamma^s(u)^*|\ge |u|$. Indeed, we can write 
$u=r_1w_1 + \dots + r_kw_k$ for some homogeneous $r_i\in R$ with 
$|r_i|+|w_i|=|u|$ for each $i$. Thus 
$\gamma(u)= \sum \tau(r_i)\gamma(w_i)$ and hence 
$|\gamma(u)^*| \ge \min\{ |\tau(r_i)| + |\gamma(w_i)^*| \} = |u|$. 
Now if $|\gamma^s(u)^*|\ge |u|$ for some $s\ge 1$, then writing 
$\gamma^s(u) = \sum_{i\ge |\gamma^s(u)^*|}u_i$ for homongeneous $u_i$ with 
$|u_i|=i$, we obtain 
$|\gamma^{s+1}(u)^*|\ge \min\{|\gamma(u_i)^*|\} \ge \min\{|u_i|\} \ge |u|$, 
hence the desired inequality follows by induction.   

Next, suppose again $u$ is homogeneous.  
If $|\gamma^s(u)^*| > |u|$ for some $s\ge 1$ then, writing 
$\gamma^s(u) = \sum_{i\ge |\gamma^s(u)^*|} u_i$ for some homogeneous 
$u_i$ of $|u_i|=i$, we obtain 
$\gamma^{s+1}(u) = \sum_{i\ge |\gamma(u)^*|} \gamma(u_i)$ hence 
$
|\gamma^{s+1}(u)^*| \ge 
\min\{|\gamma(u_i)^*|\} \ge 
\min\{|u_i|\} = |\gamma^s(u)^*| > |u|.
$ 
Thus induction gives us that if $|\gamma^t(u)^*|>|u|$ for some $t\ge 1$ 
then $|\gamma^s(u)^*|> |u|$ for every $s\ge t$. This  
contradicts the fact that $\gamma^m(u)=u$. It follows that  
$|\gamma^s(u)^*| = |u|$ for every $s\ge 1$. 

Finally, suppose $u\ne 0$ is arbitrary, and write it as a sum 
of homogeneous components $u=\sum_i u_i$ with $|u_i|=i$. We get 
$\gamma^s(u) = \sum_{i\ge |u^*|} \gamma^s(u_i)$, and since 
$|\gamma^s(u_i)^*| = |u_i| = i$ we must have $|\gamma^s(u)^*| = |u^*|$ 
as desired.    
\end{proof}

\begin{lemma}\mlabel{L:lowest-degrees} 
Let $\Bbbk$ be any field, let $\tau\in D_n$, and 
let $\psi\: \cx F \lra \tau_*\cx F$ be a morphism of chain complexes 
of $R$-modules such that $\psi^m=\id$ for some $m\ge 1$. 
Then, for each nonzero element $x$ 
of $H$, $G$, and $F$, we have that $|\psi(x)^*|=|x^*|$. 
\end{lemma}

\begin{proof} 
Let $\pi\: H \lra \coker\phi_1 = \HH_0\cx F$ be the canonical 
projection, and write $\varphi\: \HH_0\cx F \lra I$ for the 
isomorphism of $R$-modules induced by $\phi_0$ so that 
$\phi_0=\varphi\pi$.  In particular, we have that  
$\varphi\: \tau_*\HH_0\cx F \lra \tau_* I$ is also an isomorphism 
of $R$-modules. 
By Lemma~\ref{L:unique-action-ideal} 
we have $\varphi\ol\psi\varphi^{-1}=c\tau$ for some $c\in\Bbbk$, 
where $\ol{\psi}$ is the isomorphism 
on homology induced by $\psi$.  
The result for $x=h_i$ with $0\le i\le n$ is now immediate from the 
fact that $\ker(\phi_0)=\im(\phi_1)$ is generated in degrees $\ge 2n+1$. 
Since the $h_i$s generate $H$, the result for the other nonzero 
elements $x$ of $H$ follows from Lemma~\ref{L:graded-lowest-degrees}. 

Next let $x=g_0$ or $x=g_{i,i+1}$ or $x=g_{n,1}$. Consider the decomposition 
$\psi(x) = \sum_{i\ge |\psi(x)^*|} u_i$ into homogeneous components with $|u_i|=i$. 
Then, since $\phi_1$ preserves degrees, 
$\phi_1\psi(x)=\sum_{i\ge |\psi(x)^*|}\phi_1(u_i)$ is a decomposition of 
$\psi\phi_1(x)$ into homogeneous components. We already have shown that 
$|\psi\phi_1(x)^*|=|\phi_1(x)| = |x|\le 2n+2$, and since 
$\ker(\phi_1)=\im(\phi_2)$ is generated in degree $2n+2$, we must have 
$u_i=0$ for $i< |x|$ and $u_{|x|}\ne 0$. Thus $|\psi(x)^*|=|x|$. The 
result for the other nonzero elements of $G$ is now immediate from 
Lemma~\ref{L:graded-lowest-degrees}.

Finally, the result for the nonzero elements of $F$ is immediate from  
the above by the injectivity of $\phi_2$.  
\end{proof}

\begin{lemma}\label{L:to-homogeneous-action} 
Let \/ $\Bbbk$ be any field, let $\tau\in D_n$,  and 
let $\psi\: \cx F \lra \tau_*\cx F$ be a morphism of chain complexes 
of $R$-modules such that $\psi^m=\id$ for some $m\ge 1$. 
Let $\widetilde\psi\: \cx F \lra \tau_*\cx F$ be the map defined on each  
basis element $x$ of $H$, $G$, and $F$ by 
$\widetilde\psi(x) = \psi(x)^*$. 

Then $\widetilde\psi$ is an isomorphism of complexes of graded $R$-modules, 
and $\widetilde\psi^m=\id$.  
\end{lemma}

\begin{proof} 
Let $x$ be a basis element of $H$, $G$, or $F$. Then $\phi(x)$ is 
a nonzero homogeneous element of $\cx F$ of same degree as $x$. We 
know $\psi[\phi(x)]=\phi[\psi(x)]$, and by Lemma~\ref{L:lowest-degrees}
we also have the decomposition 
$\psi(x)=\sum_{i\ge |x|}u_i$ 
into homogenous components. This shows that $\sum_{i\ge |x|}\phi(u_i)$ is 
a decomposition of $\psi[\phi(x)]$ into homogeneous components. Since we 
know from Lemma~\ref{L:lowest-degrees} that 
$|\psi[\phi(x)]^*|=|\phi(x)|=|x|$, 
we get 
$
\widetilde\psi[\phi(x)]=\psi[\phi(x)]^*=\phi(u_{|x|})=
\phi[\psi(x)^*]=\phi[\widetilde\psi(x)]$. 

Next, note that by Lemma~\ref{L:graded-lowest-degrees} and 
Lemma~\ref{L:lowest-degrees} we have $|\psi^s(x)^*|=|x^*|$ for every 
nonzero element $x$ of $\cx F$ and every $s\ge 1$. Therefore 
$\psi^{s+1}(x)^*=\psi^s[\psi(x)^*]^*$ for every $s\ge 1$ and every nonzero 
$x$. In particular, for every homogeneous basis 
element $x$ we get $\psi^s(x)^*=\psi^{s-1}[\widetilde\psi(x)]^*$. 
Iterating this we obtain  
$
x=\psi^n(x)=\psi^n(x)^*=\psi^{n-1}[\widetilde\psi(x)]^* =    
  \psi^{n-2}[(\widetilde\psi)^2(x)]^* 
= \dots = \psi[\widetilde\psi^{n-1}(x)]^* = \widetilde\psi^n(x)
$.  
\end{proof}

\begin{lemma}\label{L:homogeneous-action-formulas}
Let $\Bbbk$ be an algebraic extension of\/ $\mathbb F_p$, and 
let $\psi\: \cx F \lra \rho_*\cx F$ be a morphism of chain complexes 
of graded $R$-modules such that $\psi^n=\id$.  
Then the diagram 
\begin{equation}\label{E:unique-action-complex} 
\begin{CD}
H       @> \psi_0 >> \rho_* H \\  
@V \phi_0 VV              @VV \phi_0 V    \\ 
I       @>> \rho > \rho_* I 
\end{CD}
\end{equation}
is commutative, 
and  
we have 
\begin{equation}\label{E:action-formulas}
\begin{split} 
\psi(h_i) &= h_{\rho(i)} \quad\text{for } 0\le i\le n; \\ 
\psi(g_{i,j})&=g_{\rho(i),\rho(j)}  \\ 
\psi(g_0) &= g_0 + a_1g_{1,2} + \dots + a_ng_{n,1}; \\  
\psi(f)&=f
\end{split}
\end{equation}
where the $a_i$ are homogeneous elements of $R$ of degree $1$. 
\end{lemma}

\begin{proof} 
Let $\pi$ and $\varphi$ be as in the proof of 
Lemma~\ref{L:lowest-degrees}. 
We obtain therefore an isomorphism of $R$-modules 
$\widehat\psi=\varphi\ol\psi\varphi^{-1}\: I \lra \rho_* I$,    
where $\ol{\psi}$ is the isomorphism 
on homology induced by $\psi$. 
Thus $\widehat\psi = c\rho$ for some $c\in\Bbbk$ by \
Lemma~\ref{L:unique-action-ideal}. 
Since by Remark~\ref{R:trivial-action}(b)
$\rho$ acts trivially on $\Bbbk$,
this yields
$
\id = \widehat{\psi}^n = c^n\rho^n = c^n\id, 
$
hence $c^n=1$ and, as $n$ is a power of $p$, we get $c=1$.  
It follows that \eqref{E:unique-action-complex} commutes 
as desired.  

Next, since $\psi$ induces an isomorphism on the homogeneous components 
of $H$ of degrees $2n-1$ and $2n$, and since $\phi_0$ is injective on these 
components, the desired equalities $\psi(h_i)=h_{\rho(i)}$ for $i=0,\dots,n$ 
are immediate from the commutativity of \eqref{E:unique-action-complex}. 
Since $\phi_1\psi(g_{i,j}) = \psi\phi_1(g_{i,j})=\phi_1(g_{\rho(i),\rho(j)})$, 
the injectivity of $\phi_1$ in degree $2n+1$ yields that 
$\psi(g_{i,j})=g_{\rho(i),\rho(j)}$ as desired. Now we have 
$\phi_2\psi(f)=\psi\phi_2(f)=\phi_2(f)$, therefore  
$\psi(f)=f$ by injectivity of $\phi_2$.  

Finally, since $\psi$ preseves degrees, we must have 
$\psi(g_0)=cg_0 + h'$ where $h'=a_1g_{1,2}+\dots +a_ng_{n,1}$ 
for some homogeneous 
$a_i\in R$ of degree $1$ and some constant $c\in \Bbbk$. 
Since by Remark~\ref{R:trivial-action}(b)
$\rho$ acts trivially on $\Bbbk$, it follows that 
$g_0 = \psi^n(g_0) = c^ng_0 + \sum_{k=1}^nc^{n-k}\psi^{k-1}(h')$. By what we 
already proved we know that each $\psi^{k-1}(h')$ 
is again a homogeneous linear combination of the $g_{i,j}$s,  
and therefore $c^n=1$ hence $c=1$ as $n$ is a power of $p$. 
\end{proof}

\begin{proof}[Proof of Theorem~\ref{T:counterexample}]
It suffices to show that if $\Bbbk$ is an algebraic field extension 
of $\mathbb F_p$ then 
the action on $I$ of the cyclic group $C_n$ generated by 
$\rho$ in $D_n$ cannot be lifted to an action on $\cx F$. Indeed, 
suppose we have an isomorphism $\psi\: \cx F\lra \rho_*\cx F$ 
such that $\psi^n=\id$. Then by 
Lemma~\ref{L:to-homogeneous-action} and 
Lemma~\ref{L:homogeneous-action-formulas}, we may assume that 
$\psi$ acts via the equations \eqref{E:action-formulas}. 
Since by Remark~\ref{R:trivial-action}(b)
$\rho$ acts trivially on $\Bbbk$, going modulo the 
augmentation ideal $J$ in $R$ generated by the set $\{y-1 \mid y\in X \}$, 
we obtain an action of the group ring $\Bbbk[C_n]$ on the chain 
complex of $\Bbbk$-vector spaces $\overline{\cx F}= \cx F/J\cx F$.  
We see that $\Bbbk[C_n]\bar\phi_1(\bar g_0)$ is $n$-dimensional 
over $\Bbbk$, hence is a free $\Bbbk[C_n]$-module, therefore 
$\Bbbk[C_n]\bar g_0$ is also a free $\Bbbk[C_n]$-module, in particular 
the element $w=(1 + \rho + \dots + \rho^{n-1})\bar g_0$ 
must be a non-zero 
$C_n$-invariant element of the $\Bbbk[C_n]$-submodule $G'$ of $\overline G$ 
generated over $\Bbbk$ by the elements $\bar g_{i,j}$. Since the 
only invariants of $G'$ are the $\Bbbk$-multiples of $\bar\phi_2(\bar f)$, 
we see that $w$ is a non-zero element in $\Ker(\phi_1)$, hence 
$\Bbbk[C_n]\phi_1(\bar g_0)$ cannot be a free $\Bbbk[C_n]$-module, 
a contradiction.   
\end{proof}

\end{document}